\documentclass[11pt]{article}

\usepackage{latexsym,mathrsfs}
\usepackage{amsmath,amssymb} 
\usepackage{amsthm,enumerate,verbatim}
\usepackage{amsfonts}
\usepackage{graphicx}
\usepackage{algorithm}
\usepackage{algorithmic}
\usepackage{url}
\usepackage{xcolor}
\usepackage[version=4]{mhchem}

\makeatletter
\newsavebox\myboxA
\newsavebox\myboxB
\newlength\mylenA
\newcommand*\xoverline[2][0.75]{%
    \sbox{\myboxA}{$\m@th#2$}%
    \setbox\myboxB\null
    \ht\myboxB=\ht\myboxA%
    \dp\myboxB=\dp\myboxA%
    \wd\myboxB=#1\wd\myboxA
    \sbox\myboxB{$\m@th\overline{\copy\myboxB}$}
    \setlength\mylenA{\the\wd\myboxA}
    \addtolength\mylenA{-\the\wd\myboxB}%
    \ifdim\wd\myboxB<\wd\myboxA%
       \rlap{\hskip 0.5\mylenA\usebox\myboxB}{\usebox\myboxA}%
    \else
        \hskip -0.5\mylenA\rlap{\usebox\myboxA}{\hskip 0.5\mylenA\usebox\myboxB}%
    \fi}

\newtheorem{lemma}{Lemma}

\newtheorem{theorem}{Theorem}

\newtheorem{corollary}{Corollary}
\newtheorem{definition}{Definition}
\newtheorem*{definition*}{Definition}
\newtheorem{example}{Example}

\newtheorem*{assumption*}{Assumption}
\newtheorem{remark}{Remark}

\DeclareMathOperator{\conv}{conv} 
\DeclareMathOperator{\supp}{supp}

\DeclareMathOperator{\cone}{cone} 
\DeclareMathOperator{\rank}{rank}

\DeclareMathOperator{\col}{col}

\definecolor{brightpink}{rgb}{1.0, 0.0, 0.5}

\newcommand{\revise}[1]{{{\color{black} #1}}}

\title{Partial Identifiability for Nonnegative Matrix Factorization} 
\date{}

\author{Nicolas Gillis\thanks{Department of Mathematics and Operational Research, 
University of Mons, 
Rue de Houdain~9, 7000 Mons, Belgium. 
Email: nicolas.gillis@umons.ac.be. NG acknowledges the support by the Fonds de la Recherche Scientifique - FNRS and the Fonds Wetenschappelijk Onderzoek - Vlanderen (FWO) under EOS Project no O005318F-RG47, and by the Francqui Foundation.} \and R{\'o}bert Rajk{\'o}\thanks{Institute of Mathematics and Informatics, Faculty of Sciences, University of P\'ecs, Ifj\'us\'ag u. 6. P\'ecs, Hungary H-7624,  Email: rajko@gamma.ttk.pte.hu} 
	}

\begin{document}

\maketitle

\begin{abstract}
Given a nonnegative matrix factorization, $R$, and a factorization rank, $r$, Exact nonnegative matrix factorization (Exact NMF) decomposes $R$ as the product of two nonnegative matrices, $C$ and $S$ with $r$ columns, such as $R = CS^\top$. A central research topic in the literature is the conditions under which such a decomposition is unique/identifiable, up to trivial ambiguities. In this paper, we focus on partial identifiability, that is, the uniqueness of a subset of columns of $C$ and $S$. We start our investigations with the data-based uniqueness (DBU) theorem from the chemometrics literature. The DBU theorem analyzes all feasible solutions of Exact NMF, and relies on sparsity conditions on $C$ and $S$. We provide a mathematically rigorous theorem of a recently published restricted version of the DBU theorem, relying only on simple sparsity and algebraic conditions: it applies to a particular solution of Exact NMF (as opposed to all feasible solutions) and allows us to guarantee the partial uniqueness of a single column of $C$ or $S$. Second, based on a geometric interpretation of the restricted DBU theorem, we obtain a new partial identifiability result. This geometric interpretation also leads us to another partial identifiability result in the case $r=3$. Third, we show how partial identifiability results can be used sequentially to guarantee the identifiability of more columns of $C$ and $S$. We illustrate these results on several examples, including one from the chemometrics literature. 
\end{abstract}


\section{Introduction}

Given a nonnegative matrix $R \in \mathbb{R}^{m \times n}_+$ and a factorization rank $r$, nonnegative matrix factorization (NMF) requires to compute two nonnegative matrices, 
$C \in \mathbb{R}^{m \times r}_+$ and 
$S \in \mathbb{R}^{n \times r}_+$, such that $CS^\top \approx R$. 
NMF has become a standard technique in unsupervised data analysis, and has found numerous applications, e.g., in hyperspectral imaging, audio source separation, topic modeling, community detection, to cite a few; see, e.g., the books~\cite{cichocki2009nonnegative, gillis2020nonnegative} and the references therein. 
An application where NMF has been particularly popular is  multivariate curve resolution (MCR) and self-modeling curve resolution (SMCR) where the input matrix $R$ represents the total response values from some chemical measurements of mixed samples. An example is when we consider the evolution of the spectral profile of a chemical reaction over time. 
More precisely, the $i$th row of $R$ is the cumulative spectral content of the chemical reaction at the $i$th time step.
An NMF of $R$, with   
$R(i,:) \approx C(i,:) S^\top$ for all $i$,  provides the spectral signature of the chemical compounds in $C$, along with their proportion in the reaction over time in $S$. 
In general, the matrix $C$ can be considered as the composition profile-matrix (each column of matrix $C$ is a composition profile of a chemical, e.g., in a reaction in time), and the matrix $S$ is the signal profile-matrix (each column of $S$ will be the spectrum of a chemical). 
This model can cover most types of nonnegative measurement matrices, and has been used successfully in chemistry, physics, biology, engineering, and informatics~\cite{malinowski2002facanalchem,common2010blindsource,cichocki2009nonnegative,gillis2020nonnegative,Wehrens2020R,brown2020comprchemom, leplat2021multiplicative}. 
We provide a consecutive reaction example in which the reactant $X$ forms an intermediate $Y$ and the intermediate forms the product $Z$ in two irreversible first-order reactions:
\ce{ $X$ ->C[$k_1$=20] $Y$ ->C[$k_2$=3] $Z$ }, where $k_1$ and $k_2$ are the first and the second reaction rate constants, respectively.  Figure~\ref{fig:consreac} depicts the data matrix curves and the original composition and signal profiles for the three components, $X$ in `Navy Blue', 
$Y$ in `Chocolate', and 
$Z$ in `Gold Tips' colors. See Section~\ref{sec:numexp} for another example. 
\begin{figure}[ht!]
\begin{center}
\includegraphics[width=1\textwidth]{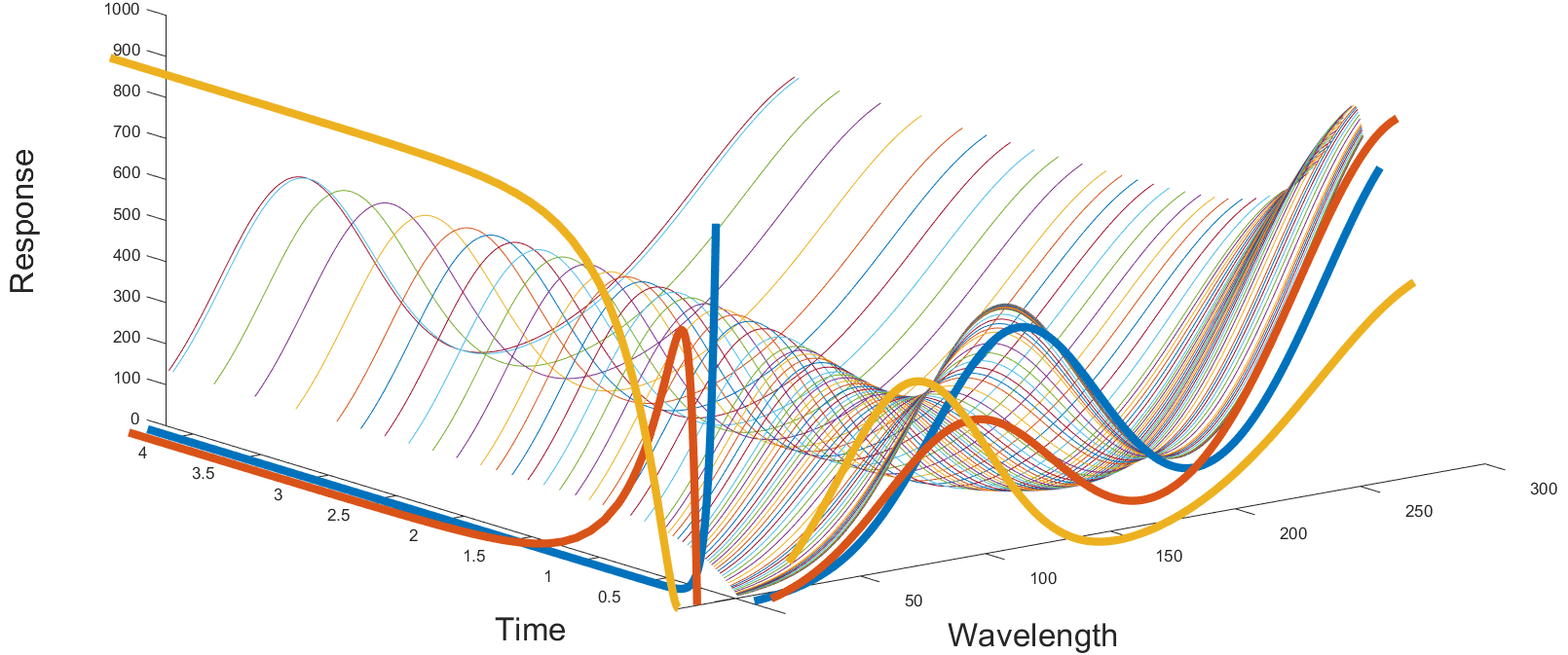} 
\caption{A three-component consecutive reaction example with the original composition in time, and the signal in wavelength profiles. \label{fig:consreac}} 
\end{center}
\end{figure}


\paragraph{Uniqueness / Identifiability} 

A crucial question in many applications is the uniqueness of a decomposition $CS^\top$, up to permutation and scaling, which is also known as the identifiability of $CS^\top$. In fact, uniqueness/identifiability (we will use both words interchangeably without attributes) allows NMF to recover the groundtruth factors that generated the data, such as the sources in audio source separation,  the materials in hyperspectral images, and the chemical components in a reaction; 
see the discussions in~\cite{xiao2019uniq} and~\cite[Chap.~4]{gillis2020nonnegative}, and the references therein.  
To attack this question, we focus in this paper on Exact NMF (that is, an errorless reconstruction), defined as follows. 
\begin{definition}[Exact NMF of size $r$] 
Given a nonnegative matrix $R \in \mathbb{R}^{m \times n}$, 
the decomposition $CS^\top$ where $C \in \mathbb{R}^{m \times r}_+$ and 
$S \in \mathbb{R}^{n \times r}_+$ is an exact NMF of $R$ of size $r$ if $R = CS^\top$. 
\end{definition}

Let us formally define the full uniqueness/identifiability of an Exact NMF.  
\begin{definition}[Full identifiability of Exact NMF] \label{def:uniqNMF} 
The Exact NMF of $R = C_\star S_\star^\top$ of size $r$ is (fully) identifiable (also known as unique, or essentially unique) if and only if for any other Exact NMF of $R = CS^\top$ of size $r$, there exists a permutation matrix $\Pi \in \{0,1\}^{r \times r}$ and a nonsingular diagonal scaling matrix $D$ such that 
\[
C = C_\star \Pi D 
\quad 
\text{ and } 
\quad  
S^\top =  D^{-1} \Pi^\top S_\star^\top. 
\]
In other words, any other Exact NMF of $R = CS^\top$ of size $r$ has the form  
\begin{equation} \label{eq:scaleperm}
CS^\top  
= 
 \sum_{k=1}^r 
 C(:,k) S(:,k)^\top  
 \; 
 =
 \; 
 \sum_{k=1}^r 
 \underbrace{ \alpha_k C_\star(:,\pi_k) }_{C(:,k)} 
 \underbrace{\frac{1}{\alpha_k} S_\star(:,\pi_k)^\top}_{S(:,k)^\top}, 
\end{equation}   
for some permutation $\pi$ of $\{1,2,\dots,r\}$, 
and  some positive scalars $\alpha_k$ $(1 \leq k \leq r)$. 
\end{definition} 

In the NMF literature, all works we are aware of have focused on the full identifiability of Exact NMF, and actually this is simply referred to as the  identifiability of Exact NMF. 
The chemometrics literature has been interested in the question of partial identifiability: when all the chemical components are not identifiable, it asks whether a subset of the profiles of these chemical components are identifiable. 
In the chemometrics literature that studies the MCR problem, 
the following definitions are used~\cite{rajko2015definition}: 
\begin{itemize}
    \item full uniqueness: all profiles  of all components  are unique, that is, all columns of $C$ and $S$ are identifiable. 
    This coincides with Definition~\ref{def:uniqNMF} above. 
    
    \item partial uniqueness: both profiles of one or more, but not all, components are unique.
    
    \item fractional uniqueness: a single profile of a component is recovered uniquely while the others are not necessarily. This coincides with Definition~\ref{def:partialUniqNMF} below. 
    
    \item non-uniqueness: no profile is identifiable, that is, a unique solution does not exist even for a single profile. 
    This definition would be different from the NMF literature, where   non-uniqueness means that at least one profile is not identifiable. 

\end{itemize}

In this paper, we will focus on full identifiability  (Definition~\ref{def:uniqNMF}) and   partial identifiability which we  define as follows.   
\begin{definition}[Partial identifiability in Exact NMF] \label{def:partialUniqNMF} 
Let $R=C_\star S_\star^\top$ be an exact NMF of $R$ of size~$r$. 
The $k$th column of $C_\star$ is identifiable if and only if for any other Exact NMF of $R = CS^\top$ of size $r$, there exists an index set $j$ and a scalar $\alpha > 0$ such that 
\[
C(:,j)  \; = \;  \alpha C_\star(:,k) . 
\]
\end{definition} 

Similarly, we can define the identifiability of the $k$th column of $S_\star$ using symmetry, which is referred to as the duality principle in the chemometrics literature~\cite{rajko2006duality}, 
since $R = C_\star S_\star^\top$ if and only if 
$R^\top = S_\star C_\star^\top$. We will focus in this paper on the partial identifiability of the first factor, $C_\star$, without loss of generality, by symmetry of the problem: any result that applies to $C_\star$ applies to $S_\star$.  

Most results on the identifiability of Exact NMF focus on the case  $r = \rank(R)$ as it is the most reasonable in most applications. We will also focus on this case in this paper.

\paragraph{Contribution and outline of the paper} Although partial  identifiability has been considered in the chemometrics literature, there does not exist, to the best of our knowledge, a detailed formal description (that is, a formal mathematical theorem) of the assumptions needed to obtain such results, nor rigorous proofs. 
The main contribution of this paper is to provide several new theorems regarding the partial identifiability of Exact NMF. 

The paper is organized as follows. 
In Section~\ref{sec:prelimgeo}, we briefly recall the geometric interpretation of Exact NMF on which our results and many identifiability results in the literature  rely on. 
In Section~\ref{sec:prev}, we review important results on the identifiability of Exact NMF that will be useful in our discussions.   
Section~\ref{sec:first} contains our main contributions, namely 
\begin{itemize}
    \item The restricted DBU theorem (Theorem~\ref{mainth}), a partial identifiability theorem for Exact NMF. 
    
    \item A geometric interpretation of the restricted DBU theorem (Lemma~\ref{lemFRZRWvert}). 
    It will lead us to a new partial identifiability theorem for Exact NMF (Theorem~\ref{mainthgeo0}). 
    
    \item A new theorem allowing us to use any partial identifiability theorem sequentially to guarantee the uniqueness of several columns of $C$ and $S$ (Theorem~\ref{mainth2}).     
    
    \item A new partial identifiability theorem for Exact NMF in the special case $r=3$ (Theorem~\ref{corthgeo}). 
\end{itemize}

Finally, in Section~\ref{sec:appli}, we discuss the practical implications of our result, provide an algorithm to automatically check partial identifiability in an Exact NMF which is available from 
\begin{center}
\color{blue} \url{https://gitlab.com/ngillis/nmf-partial-identifiability} 
\end{center}
\color{black} 
along with all the examples presented in the paper,  
and illustrate the algorithm on an example from the chemometrics literature. Note that we also provide small examples throughout the paper to illustrate our theoretical results.

\section{Preliminary: Geometric interpretation of Exact NMF} \label{sec:prelimgeo}

Most results on the identifiability of Exact NMF rely on its geometric interpretation, including the results of this paper. 
We therefore briefly recall it here for completeness. 

For an Exact NMF $R = CS^\top$, we can assume w.l.o.g.\ that $R$, $C$ and $S^\top$ are column stochastic, that is, the entries in each column sum to one. Hence each column of $R$, $C$ and $S^\top$ has unit $\ell_1$-norm (a.k.a.\ absolute sum norm, area norm, grid norm, taxi cabnorm, Manhattan norm). The $\ell_1$-norm coincides with the so-called Borgen norm with $z = e$ in the chemometrics literature, where $e$ is the vector of all ones of appropriate dimension~\cite{GRANDE2000,rajko2009BorgenNorm}. 
In fact, one can first remove zero columns and rows of $R$, and remove the corresponding columns and rows of $S^\top$ and $C$, respectively, 
which do not bring any useful information, while it may lead to numerical problems~\cite{Omidikia2020sparse-norms}. 
Then one can normalize $R = CS^\top$ as follows: 
 \begin{equation} \label{eq:sumtoone}
 R_n(:,j) 
 :=  \frac{R(:,j)}{R(:,j)^\top e}  
 = \sum_{k=1}^r \underbrace{
 \frac{C(:,k)}{C(:,k)^\top e}
 }_{
  := C_n(:,k)
 }
    \underbrace{
    \frac{C(:,k)^\top e}{R(:,j)^\top e} S(j,k)
    }_{
     := S_n(j,k)
    }
    = \sum_{k=1}^r C_n(:,k) S_n(j,k). 
 \end{equation}
Hence $R_n = C_n S_n^\top$ where $R_n$ and $C_n$ are column stochastic (that is, $e^\top = e^\top R_n$ and $e^\top C_n=e^\top$), by construction, while $S_n^\top$ is because 
\begin{equation} \label{sumtoonelemma} 
e^\top = e^\top R_n = e^\top C_n S_n^\top = e^\top S_n^\top. 
\end{equation}
Let us therefore assume, w.l.o.g., that $R$, $C$ and $S^\top$ are column stochastic. See the chemometrics analogue using Borgen norms and closure in~\cite{rajko2009BorgenNorm,Omidikia2018closure}. 
 This means that, after normalization, the columns of $R$ belong the convex hull of the columns of~$C$ that are column stochastic, since, for all $j$,  
 \[
 R(:,j) = \sum_{k=1}^r C(:,k) S(j,k) = C S(j,:)^\top, 
 \]
 where 
 $S(j,:)^\top \in \Delta = \{ x \ | \ x \geq 0, e^\top x = 1\}$, with  $\Delta$ the probability simplex of appropriate dimension. 
In the case $r = \rank(R)$, we must have $\col(R) = \col(C)$, and therefore 
 \[
 \conv(R) 
 \;  \subseteq \;  \conv(C) 
 \; \subseteq \;  \Delta \cap \col(R),  
 \]
 where   $\conv(R) = \{ x \ | \ x = Ry, y \in \Delta \}$; 
 see, e.g., \cite[Chapter~2]{gillis2020nonnegative}. 
Hence Exact NMF reduces to finding a polytope (that is, a bounded polyhedron), $\conv(C)$ with $r$ vertices (the columns of $C$), nested between $\conv(R)$ and $\Delta \cap \col(R)$. This is the so-called nested polytope problem (NPP) in  computational geometry which is defined as follows. 
 \begin{definition}[Nested polytope problem (NPP)]
 Given a full-dimensional inner polytope defined by its vertices $\{v_1,v_2,\dots,v_n\}$, that is, 
 \[
 \mathcal{P}_{inn} = \conv([v_1,v_2,\dots,v_n]) \subseteq \mathbb{R}^d, 
 \] 
 a full-dimensional outer polytope defined by its facets\footnote{A facet of 
a $d$-dimensional polytope is a $(d-1)$-dimensional face. A face of a polytope is the intersection of that polytope with any closed half space whose boundary is disjoint from the interior of the polytope. 
For the polytope $\mathcal{P}_{out}$, each facet will have the form 
$\{  x \in  \mathcal{P}_{out} \ | \ F(i,:) x + g_i = 0 \}$ for some $i$. } 
 \[
 \mathcal{P}_{out} = \{ x  \in \mathbb{R}^d \ | \ Fx + g \geq 0 \} \text{ where } 
 F \in \mathbb{R}^{m \times r} \text{ and }  g \in \mathbb{R}^{m}, 
 \] 
such that $\mathcal{P}_{inn} \subseteq \mathcal{P}_{out}$, and an integer $p \geq d+1$, 
 find a polytope, $\mathcal{P}_{bet}$, with $p$ vertices nested \emph{between}  $\mathcal{P}_{inn}$ and $\mathcal{P}_{out}$, that is, $\mathcal{P}_{inn} \subseteq  \mathcal{P}_{bet}  \subseteq  \mathcal{P}_{out}$.  
 \end{definition}
 
 The polytope  $\conv(R)$ is typically not full-dimensional, since $m > r$ in most cases. In fact, $\conv(R)$ has dimension $\rank(R) - 1$, and the NPP corresponding to the Exact NMF of $R$ satisfies $d = \rank(R) - 1$. 
 However, up to restricting the solution space to the affine hull of $R$, 
 the set $\conv(R)$ plays the role of $\mathcal{P}_{inn}$ in the NPP, and $\Delta \cap \col(R)$ the role of $\mathcal{P}_{out}$. 
 
 \begin{theorem}\cite{vavasis2010complexity}   
The Exact NMF problem with $r = \rank(R)$ is equivalent to an NPP with $d = \rank(R)-1$ and $p = d+1$, and vice versa.  
 \end{theorem} 
 
 The equivalence between Exact NMF and the NPP can be used to study the identifiability of Exact NMF. 
 For example, for $\rank(R) = r = 2$, the NPP is trivial since $\mathcal{P}_{inn}$ and $\mathcal{P}_{out}$ are one-dimensional polytopes, that is, segments~\cite{lawton1971self}; see also, e.g., \cite{Rajko2010additional2comp}. 
 The Exact NMF of $R$ when $r = 2$ is unique if and only if $\mathcal{P}_{inn} = \mathcal{P}_{out}$ in the corresponding NPP, which leads to necessary and sufficient conditions of $R$; see Section~\ref{sec:prev}.  
 
 For $r=3$, the NPP has dimension two, and has been used extensively in the MCR literature to study the identifiability of Exact NMF; 
 see, e.g., \cite{Borgen1985compres3comp,Rajko2005analsol3comp,Golshan2016review3compRFS}.
They refer to the NPPs with feasible regions as Borgen-Rajk\'o plots; 
see below for an example of NPPs, and also Section~\ref{sec:first}. 
In this case, it is particularly useful to know how to reduce an instance of Exact NMF to the NPP, and vice versa. Let us briefly recall these reductions which we will use later in the paper. 

\paragraph{From Exact NMF to the NPP} Let $R$ be an instance of Exact NMF with $r = \rank(R)$. First remove zero columns and rows of $R$, 
and normalize $R$ to become column stochastic. Let $\mathcal{L}$ be the index set of $r$ linearly independent columns of $R$, so that $R = R(:, \mathcal{L}) V \geq 0$ for some $V$. Since $R$ and $U = R(:, \mathcal{L})$ are column stochastic, \revise{the entries in each column of $V$ sum to one, by the same argument as in~\eqref{sumtoonelemma}}. 
We define $v_j = V(1:r-1,j)$ for $j=1,2,\dots,n$ whose convex hull form $\mathcal{P}_{inn}$, while   
\[
\mathcal{P}_{out} = \{ x \in \mathbb{R}^{r-1} 
\ | \ 
U(:,1:r-1) x + U(:,r) (1-e^\top x) \geq 0 \}. 
\]
This NPP instance has a solution with $r$ vertices if and only if $R$ admits an exact NMF of size $r$~\cite{vavasis2010complexity}. 

\paragraph{From NPP to Exact NMF} This reduction is particularly useful to construct matrices coming from NPP problems in two dimensions. Given an NPP instance, the matrix $R$ is constructed as follows: for all $j=1,2,\dots,n$  
\[
R(:,j) = F v_j + g, \text{ where } R(:,j) \geq 0 \text{ since } v_j \in \mathcal{P}_{inn} \subseteq \mathcal{P}_{out}. 
\]
The matrix $R$ admits an exact NMF of size $r$ if and only if the NPP instance has a solution with $r$ vertices~\cite{vavasis2010complexity}. Observe that each row of $R$ corresponds to a facet of $\mathcal{P}_{out}$ and each column to a vertex of $\mathcal{P}_{inn}$, while $R(i,j)$ is the so-called slack of the $j$th vertex with respect to the $i$th facet, namely $R(i,j) = F(i,:) v_j + g_i$.  

 
 \begin{example} \label{exNPP}
 Let us consider the NPP where $\mathcal{P}_{out}$ is the unit square  $[0,1]^2$ defined with the inequalities $Fx + g \geq 0$ where 
 \[
 F = \left( \begin{array}{cccc}
      0  &   0 &   1  &  -1 \\
     1 &   -1  &   0   &  0  
\end{array} \right)^\top, \; 
 g = \left( \begin{array}{cccc} 0 & 1 & 0 & 1 
  \end{array} \right)^\top, 
 \] 
 while $\mathcal{P}_{inn}$ is the quadrilateral with the four vertices 
 $v_1=(0.5,0)$, $v_2=(0,0.5)$, $v_3=(0.25,0.75)$ and $v_4 = (0.75,0.25)$; 
 see Figure~\ref{ex1NPP} for an illustration. 
 \begin{figure}[ht!]
\begin{center}
\includegraphics[width=0.4\textwidth]{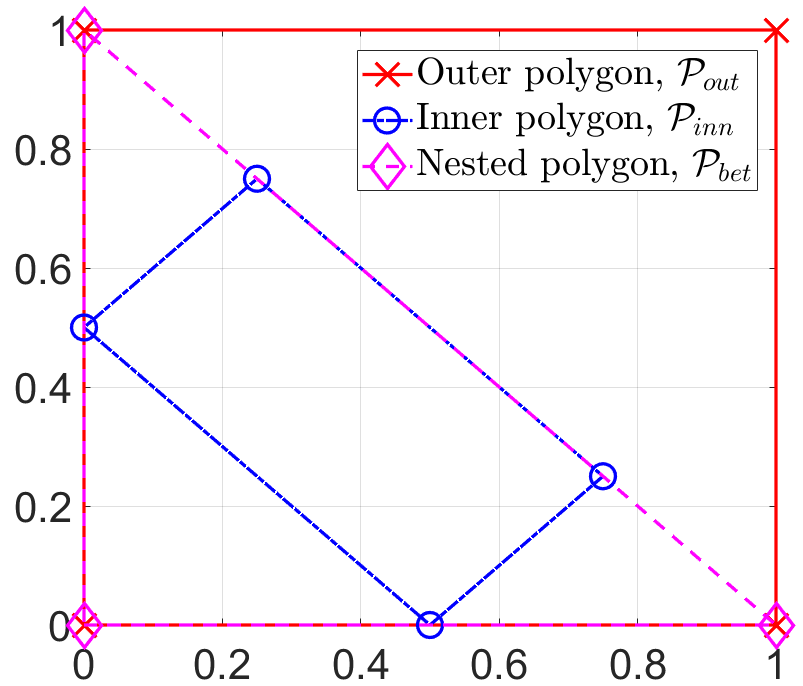} 
\caption{Illustration of the NPP instance described in Example~\ref{exNPP}.\label{ex1NPP}} 
\end{center}
\end{figure}
 The matrix $R$ of the corresponding Exact NMF problem is given by $R(:,j) = F v_j + g$ for all $j$, that is, 
 \[
 R = \frac{1}{4} \left( \begin{array}{cccc}
      0  &   2 &   3 &   1 \\
      4  &   2 &   1 &   3 \\
      2  &   0 &   1 &   3 \\
      2  &   4 &   3 &   1
 \end{array} \right). 
 \] 
 Looking at Figure~\ref{ex1NPP}, we observe that the unique nested triangle between $\mathcal{P}_{inn}$ and $\mathcal{P}_{out}$ has the vertices 
 $s_1 = (0,0)$,  $s_2 = (1,0)$ and $s_3 =  (0,1)$, implying that $R$ has a unique Exact NMF of size 3, 
 given by 
  \[
 R = \frac{1}{4} \left( \begin{array}{cccc}
      0  &   2 &   3 &   1 \\
      4  &   2 &   1 &   3 \\
      2  &   0 &   1 &   3 \\
      2  &   4 &   3 &   1
 \end{array} \right) 
 = 
 \frac{1}{4} \left( \begin{array}{ccc}
      0  &   0 &   1  \\
      1  &   1 &   0  \\
      0  &   1 &   0  \\
      1  &   0 &   1 
 \end{array} \right) 
 \left( \begin{array}{cccc}
      2  &   2 &   0 &   0  \\ 
      2  &   0 &   1 &   3 \\
      0  &   2 &   3 &   1 \\
 \end{array} \right). 
 \] 
 The first factor, $C$, in the above decomposition is obtained using $C(:,j) = Fs_j + g$ for $j=1,2,3$. 
 \end{example}

\section{Previous works on the identifiability of Exact NMF} \label{sec:prev}

The conditions that makes Exact NMF identifiable have been studied extensively in the literature. In this section, 
we briefly review some of the important works on the identifiability of Exact NMF.

\subsection{Full identifiability} \label{sec:fullid}

Let us first discuss some conditions under which Exact NMF is fully identifiable, as in Definition~\ref{def:uniqNMF}.

\paragraph{Necessary condition}

Let us state a necessary condition for Exact NMF to be fully identifiable. The condition has been rediscovered several times, and is relatively easy to prove. It is based on the support of the columns of $C_\star$ and $S_\star$; the support being the set of indices containing the non-zero entries. 
\begin{theorem} \label{th:nec} 
Let $R = C_\star S_\star^\top$ be a fully identifiable exact NMF of $R$ of size~$r$. 
Then, the support of any column of $C_\star$ (resp.\ $S_\star$)  does not contain the support of any other column of $C_\star$ (resp.\ $S_\star$). 
\end{theorem} 
\begin{proof}
If a column of $C_\star$, say $C_\star(:,1)$, contains the support of another column, say $C_\star(:,2)$, then $C_\epsilon(:,1) = C_\star(:,1) - \epsilon C_\star(:,2) \geq 0$ for $\epsilon > 0$ sufficiently small, which allows to construct another Exact NMF. In fact, taking  
$S_\epsilon(:,2) =  S_\star(:,2) + \epsilon S_\star(:,1) \geq 0$ and keeping the other columns untouched, that is, 
$C_\epsilon(:,k) = C_\star(:,k)$ for all $k \neq 1$ and $S_\epsilon(:,k) = S_\star(:,k)$ for all $k \neq 2$, we obtain an Exact NMF $C_\epsilon S_\epsilon^\top$ which is not a permutation and scaling of $C_\star S_\star^\top$. 
\end{proof}

 Interestingly, this condition is also sufficient when $r=2$, which is the only case for which we have a necessary and sufficient condition for Exact NMF to be identifiable. 
 For $r=2$, this means that Exact NMF is identifiable if and only if $C_\star$ and  $S_\star$ contain a 2-by-2 diagonal submatrix: each of the two columns of $C_\star$ (resp.\ $S_\star$) must contain a positive entry where the other column has a zero entry.

\paragraph{Sufficient condition based on separability} 

Several identifiability results for Exact NMF are based on the separability condition, defined as follows. 
\begin{definition}[Separability] \label{def:sep}
The matrix $C \in \mathbb{R}^{m \times r}$ with $m \geq r$ is separable if there exists an index set $\mathcal{K}$ of size $r$ 
such that $C(\mathcal{K},:) \in \mathbb{R}^{r \times r}$ 
is a nonsingular diagonal matrix. 
\end{definition}
Equivalently, the separable conditions requires that for each $k=1,2,\dots,r$, there exists an index~$j$ such that $C(j,:) = \alpha e_{(k)}^\top$ for some $\alpha > 0$ where $e_{(k)}$ is the $k$th unit vector (that is, the $k$th column of the identity matrix; recall that the notation $e$ without the subscript $_{(k)}$ is for the all-one vector, that is, the vector of all ones of appropriate dimension). 
The separability condition was introduced in the NMF literature by Donoho and Stodden~\cite{donoho2004does}. 
We have the following result. 
\begin{theorem} \label{th:sep} 
Let $R = C_\star S_\star^\top$ be an exact NMF of $R$ of size~$r$. 
If $C_\star$ and $S_\star^\top$ are separable, then $R = C_\star S_\star^\top$ is fully identifiable. 
\end{theorem} 
It is difficult to trace back the origin of Theorem~\ref{th:sep}, and it does explicitly appear in~\cite{donoho2004does}, although it can be derived from their result, where they relax the condition on $C_\star$. 
Other sufficient conditions based on separability have been proposed in the literature, where $S_\star$ is required to be separable, while there is some sparsity conditions on $C_\star$~\cite{laurberg2008theorems}. The results of this paper will be of this flavour, but will focus on partial identifiability.

\paragraph{Sufficiently scattered condition} 

The separability condition for both factors, $C$ and $S^\top$, is rather strong, and not satisfied by most data sets; see, e.g., the discussion in~\cite[Chapter 4]{gillis2020nonnegative}. It can be relaxed to the following condition while retaining the full identifiability. 
\begin{definition}[Sufficiently scattered condition] \label{def:SSC}
The matrix $C \in \mathbb{R}^{m \times r}$ with $m \geq r$ satisfies the sufficiently scattered condition (SSC) if 
\begin{enumerate}
  \item  $\{x \in \mathbb{R}^r_+ \ | \ e^\top x \geq \sqrt{r-1} \|x\|_2 \} \; \subseteq \; \cone(C^\top) = \{x \ | \ x = C^\top h \text{ for } h \geq 0\}$. \vspace{0.1cm} 
  
\item  There does not exist any orthogonal matrix $Q$ such that $\cone(C^\top) \subseteq \cone(Q)$, except for permutation matrices. (An orthogonal matrix $Q$ is a square matrix such that $Q^\top Q = I$.) 
\end{enumerate}
\end{definition}

Geometrically, separability requires that $\cone(C^\top)$ is the nonnegative orthant, while the SSC only requires $\cone(C^\top)$ to contain the second-order (ice-cream) cone tangent to every facet of the nonnegative orthant. 

\begin{theorem}\cite{huang2013non} \label{th:ssc} 
Let $R = C_\star S_\star^\top$ be an exact NMF of $R$ of size~$r$. 
If $C_\star$ and $S_\star^\top$ satisfy the SSC, 
then $R = C_\star S_\star^\top$ is fully identifiable. 
\end{theorem} 

It is out of the scope of this paper to discuss in details the geometric interpretation of the SSC. An important issue with the SSC is that it is NP-hard to check in general~\cite{huang2013non}. 
We refer the reader to~\cite{huang2013non}, \cite{xiao2019uniq} and~\cite[Chapter 4]{gillis2020nonnegative} for more detail. 
We will briefly compare the SSC with our conditions in Remark~\ref{rem:SSC}. 

Other full identifiability results for NMF are based on sparsity conditions; see the recent paper~\cite{abdolali2021simplex} and the references therein.

\subsection{Partial identifiability}  \label{sec:sfs}

In the MCR literature, the set of feasible solutions (SFS) of Exact NMF, a.k.a.\ the feasible regions (FRs), has been extensively studied, especially in small dimensions ($r=3,4$)~\cite{Golshan2016review3compRFS}. 
Several algorithms exist, the best-known one is the first developed for $r=2$ by Lawton and Sylvestre~\cite{lawton1971self} who introduced and coined the special term Self Modeling Curve Resolution (SMCR) for finding all feasible solutions for a matrix decomposition with the nonnegativity constraint. 
In general, the goal of MCR (and NMF) algorithms is to compute one set of particular profiles (that is, generate one solution) without considering the fact that other profiles (solutions) may exist with the same properties (namely, satisfying the same constraints and having the same objective function value). 
Rerunning several times these algorithms with different initializations can help to detect the non-uniqueness, however, in general, this process cannot generate the SFS.  
For $r=3$, after several randomized/approximate trials, Borgen and Kowalski~\cite{Borgen1985compres3comp} published an analytical solution using the tangent and the simplex rotation algorithms. These algorithms were found mathematically hard to understand and implement for non-mathematicians, e.g., chemists, thus it was not developed further although being cited in the chemometrics literature for 20 years. Rajk{\'o} and Istv{\'a}n~\cite{Rajko2005analsol3comp} revised Borgen’s study and could enlighten the concepts based on the geometry of the abstract space. Computational geometry tools (including convex hulls, Fourier-Motzkin elimination, double-description) were used for developing the algorithm to draw Borgen-Rajk{\'o} plots. The
systematic grid search method was introduced to approximate the 
SFS/FRs numerically first for two-component systems~\cite{Vosough2006gridsearch}, and subsequently it was extended for three-component systems as well~\cite{Golshan2011grid3}. Sawall
et al.~\cite{Sawall2013polinfl3} developed the polygon inflation algorithm for three-component systems as a faster and more accurate alternative to the grid search. The duality concept was first used for calculating SFS/FRs for SMCR by Beyramisoltan and her coworkers~\cite{Beyramysoltan2014BPduality}. For $r=4$, the first attempt appeared in 2013~\cite{Golshan2013} using the triangle enclosure method to approximate the boundary of the two-dimensional slices. Later in 2016, Sawall et al.~\cite{Sawall2016comp234ch} introduced the polyhedron inflation method as the generalization of the polygon inflation one. The chapter~\cite{Sawall2016comp234ch}, and a subsequent paper~\cite{Neymeyr2019NMFg} from the same research group, provided the most comprehensive summary for the SFS/FRs and related concepts up to now. See also~\cite{laursen2022sampling} for a recent sampling algorithm for larger values of~$r$, and \revise{\cite{andersons2022analytical} for an improved algorithm for the boundary curve construction along with an  
implementation}.   

\paragraph{Necessary condition for partial identifiability} 

Interestingly, the necessary condition for the identifiability for Exact NMF based on the supports of the columns of $C$ can be extended to the partial identifiability case. Note that this result is, to the best of our knowledge, not present in the literature, although it follows directly from the proof of Theorem~\ref{th:nec}. 
\begin{theorem} \label{th:necpartial} 
Let $R = C_\star S_\star^\top$ be an exact NMF of $R$ of size~$r$. 
If the $k$th column of $C_\star$ (resp.\ $S_\star$) is identifiable, then the support of the $k$th column of $C_\star$  (resp.\ $S_\star$) does not contain the support of any other column of $C_\star$ (resp.\ $S_\star$). 
\end{theorem} 
\begin{proof}
This proof is similar to that of Theorem~\ref{th:nec}.  
\end{proof}

\paragraph{Sufficient conditions for partial identifiability: DBU and restricted DBU theorems} 

In the paper~\cite{rajko2015definition}, a partial identifiability result for NMF is presented and discussed; it is called the data-based uniqueness (DBU) concept.  Data-based means there that it does not only use the estimated profiles, but also the data generated by them and all feasible profiles. 
It was formulated based on band solutions, that is, 
using not just a particular set of estimated profiles (that is, a particular Exact NMF solution), but all feasible solutions based on SMCR (that is, the corresponding NPP with feasible regions a.k.a.\ Borgen-Rajk\'o plots~\cite{Golshan2016review3compRFS}). 
Thus the SFS/FRs are needed to use the original DBU concept~\cite{rajko2015definition}. 
However, as explained in Section~\ref{sec:sfs}, there are working algorithms to get SFS/FRs only for up to  four-component systems. 
This fact inspired the development of the particular-profile DBU or Restricted DBU~\cite{lakeh2022predicting} that uses a particular solution. 
It was a step back to the profile-based concept, also used in different ways by Maeder~\cite{maeder1987evolving}, 
Malinowski~\cite{malinowski1992window}, 
and Manne~\cite{manne1995resolution}. The concept was intended to offer for practitioners (such as analytical chemists) which is why both papers~\cite{rajko2015definition,lakeh2022predicting} were published in \emph{Analytica Chimica Acta}, thus the rigorous mathematical descriptions are missing. 
In the following, the lack of the formal descriptions and proofs will be remedied. 

The idea~\cite{lakeh2022predicting} was restricted to analyze a particular solution relying on the following two conditions: Given $R = C_\star S_\star^\top$ where 
$C_\star \in \mathbb{R}^{m \times r}_+$ and 
$S_\star \in \mathbb{R}^{n \times r}_+$ with \mbox{$\rank(R) = r$},  
\begin{itemize}
    \item  (Zero-region window) There exists a row of $C_\star$, say the $i$th, such that $C_\star(i,k) = 0$ and all feasible profiles $C(i,p) > 0$ for all $p \neq k$. 

    \item (Selective window) There exists a row of $S_\star$, say the $j$th, 
    such that $S_\star(j,:) = \alpha e_{(k)}^\top$ for some $\alpha > 0$, 
    that is, $S_\star(j,k) > 0$ and all feasible profiles $S(j,p) = 0$ for all $p \neq k$. 
    
\end{itemize}

Let us comment on the two conditions above: 
\begin{enumerate}
    \item Zero-region window: This condition means that $C_\star(:,k)$ contains an entry equal to zero where all other entries of $C_\star$ in the same row are positive.  
    Geometrically, this means that $C_\star(:,k)$ is the only column of $C_\star$ on some facet of the nonnegative orthant. 
    
    \item Selective window: This condition means  that there exists a column of $R$, say the $j$th, such that $R(:,j) = \gamma C(:,k)$ for some $\gamma > 0$. In words, it   means that the $k$th column of $C$ appears, up to scaling, in the data set. This is closely related to the separability condition in the NMF literature; see the previous Section~\ref{sec:fullid}.  
    In fact, all columns of $C_\star$ satisfy the selective window condition if and only if $S_\star^\top$ is separable. 
\end{enumerate}

As was mentioned above, the restricted DBU concept in~\cite{lakeh2022predicting} does not have a formal statement nor a formal proof. 
Authors provide an informal one with explanations, focusing on the intuitions behind their result, which is more suitable for not mathematically trained practitioners.

\section{Partial identifiability theorems for Exact NMF} \label{sec:first}

In this section, we 
propose {a rigorous statement and proof for the restricted DBU concept from~\cite{lakeh2022predicting}}; see Theorem~\ref{mainth} (Section~\ref{sec:mainth}). 
In Section~\ref{sub:geoThmain}, we provide a geometric interpretation of Theorem~\ref{mainth}. This leads us to a new partial identifiability result for Exact NMF, Theorem~\ref{mainthgeo0}, in Section~\ref{sec:mainth2}. 
In Section~\ref{sec:recurDBU}, we show how to apply Theorems~\ref{mainth} and~\ref{mainthgeo0} to allow the identifiability of more than one column of $C_\star$; see Theorem~\ref{mainth2}. 
Finally, in Section~\ref{sec:partialr3}, 
we use the geometric interpretation of Theorem~\ref{mainthgeo0} to obtain a new partial identifiability result for the special case $r=3$; see Theorem~\ref{corthgeo}.

\subsection{Restricted DBU theorem} \label{sec:mainth}

Let us state and prove the restricted DBU theorem. 


\begin{theorem}[Restricted DBU theorem] \label{mainth}
Let $R = C_\star S_\star^\top$ where 
$C_\star \in \mathbb{R}^{m \times r}_+$ and 
$S_\star \in \mathbb{R}^{n \times r}_+$ with $\rank(R) = r$. 
The $k$th column of $C_\star$ is identifiable if the following two conditions hold: 
\begin{itemize}
    \item  (Full-rank zero-region window--FRZRW) 
    Let $\mathcal{I} = \{ i \ | \ C_\star(i,k) = 0  \}$ be the complement of 
    the support of the $k$th column of $C_\star$.  
The submatrix of $C_\star$ formed by the rows indexed by $\mathcal{I}$ has rank $r-1$, 
that is, $\rank\big( C_\star(\mathcal{I},:) \big) = r-1$. 
    
    \item (Selective window) There exists a row of $S_\star$, say the $j$th, such that $S_\star(j,:) = \alpha e_{(k)}^\top$ for $\alpha > 0$.  
\end{itemize}
\end{theorem}
\begin{proof} Let $R = CS^\top$ be an Exact NMF of $R$ of size $r$, that is, 
$C \in \mathbb{R}^{m \times r}_+$ and 
$S \in \mathbb{R}^{n \times r}_+$. 
We need to show that $C(:,\ell) = \beta C_\star(:,k)$ for some $\ell$ and some $\beta > 0$. 
Since $R = CS^\top = C_\star S_\star^\top$, 
and since $S_\star(j,:) = \alpha e_{(k)}^\top$ for some $\alpha > 0$ (selective window condition)
, we have 
\begin{equation} \label{decompXj}
R(:,j) 
= 
C_\star S_\star(j,:)^\top  
= 
\alpha C_\star(:,k) 
= CS(j,:)^\top  
= \sum_{p=1}^r C(:,p) S(j,p). 
\end{equation} 
Let us denote the set of indices corresponding to columns of $C$ that have zero elements in $\mathcal{I}$ as 
\[
\mathcal{P} = \{ p \ | \ C(\mathcal{I},p) = 0 \} \subseteq \{1,2,\dots,r\}, 
\] 
and $\xoverline[0.9]{\mathcal{P}} = \{1,2,\dots,r\}  \backslash \mathcal{P}$ its complement. 
By nonnegativity of all the terms in~\eqref{decompXj}, 
$S(j,p) = 0$ for all $p \in \xoverline[0.9]{\mathcal{P}}$, otherwise a zero entry of $C_\star(:,k)$ is approximated by a positive one, since $C(\mathcal{I},p) \neq 0$ for $p \in \xoverline[0.9]{\mathcal{P}}$. 
Note that $|\mathcal{P}| \geq 1$ otherwise $C_\star(:,k)$ cannot be reconstructed, since $C_\star(:,k) \neq 0$ as $\rank(C_\star) = r$. 
Below, we prove that $|\xoverline[0.9]{\mathcal{P}}| \geq r-1$, and hence  $|\mathcal{P}| \leq 1$. This will imply that $|\mathcal{P}| = 1$, that is, $\mathcal{P} = \{\ell\}$ for some $\ell$. Putting this back into~\eqref{decompXj}, this gives 
 \[
\alpha C_\star(:,k) = 
\sum_{p \in \mathcal{P} = \{\ell\}} C(:,p) S(j,p) 
+ 
\sum_{p \in \xoverline[0.9]{\mathcal{P}}} 
C(:,p) \underbrace{S(j,p)}_{=0}  
= 
C(:,\ell) S(j,\ell), 
 \] 
 where $S(j,\ell) > 0$ since $C_\star(:,k) \neq 0$ and $\alpha > 0$. 
 Finally, 
 $C(:,\ell) = \frac{\alpha}{S(j,\ell)} C_\star(:,k)$ which completes the proof. 

It remains to show that $|\xoverline[0.9]{\mathcal{P}}| \geq r-1$. 
For this, let us show that the rank of $R(\mathcal{I},:)$ is $r-1$. First, note that $\rank(R) = \rank(C_\star S_\star) = \rank(C_\star) = \rank(S_\star) = r$, by the conditions that $\rank(R) = r$, $R = C_\star S_\star^\top$, both $C_\star$ and $S_\star$ have $r$ columns.  
Then, 
\[
R(\mathcal{I},:) 
= C_\star(\mathcal{I},:) S_\star^\top 
= \sum_{p \neq k} C_\star(\mathcal{I},p) S_\star(:,p)^\top  
= C_\star(\mathcal{I},\mathcal{K}) S_\star(:,\mathcal{K})^\top, 
\] 
where $\mathcal{K} = \{1,2,\dots,r\} \backslash \{k\}$. 
By the FRZRW condition,  $\rank\big( C_\star(\mathcal{I},\mathcal{K}) \big) = r-1$, while we have  $\rank\big( S_\star(:,\mathcal{K}) \big) = r-1$ since it is made of $r-1$ columns of $S_\star$ that has rank $r$. 
Since both factors in the decomposition 
$R(\mathcal{I},:)  = C_\star(\mathcal{I},\mathcal{K}) S_\star(:,\mathcal{K})^\top$ have full rank $r-1$, 
$\rank \big( R(\mathcal{I},:) \big) = r-1$. 
Now, since $R = CS^\top$, we also have 
\[
R(\mathcal{I},:) 
= C(\mathcal{I},:) S^\top  
= 
C(\mathcal{I},{\mathcal{P}} ) 
S(:, {\mathcal{P}} )^\top 
+ 
C(\mathcal{I},\xoverline[0.9]{\mathcal{P}} ) 
S(:,\xoverline[0.9]{\mathcal{P}})^\top
= 
C(\mathcal{I},\xoverline[0.9]{\mathcal{P}} ) 
S(:,\xoverline[0.9]{\mathcal{P}})^\top, 
\] 
since $C(\mathcal{I},\mathcal{P}) = 0$, by definition. As shown above, 
$\rank\big( R(\mathcal{I},:) \big) = r-1$.   
This implies that $C(\mathcal{I},\xoverline[0.9]{\mathcal{P}} )$ has at least $r-1$ columns, that is, $|\xoverline[0.9]{\mathcal{P}}| \geq r-1$. 
\end{proof} 

Let us illustrate Theorem~\ref{mainth} on a simple example. 

\begin{example} \label{ex1revis}
Let us consider 
\[ 
R = 
\underbrace{\left( 
\begin{array}{ccc}
      2 & 2 & 2   \\
      1 & 3 &  1 \\ 
      1 & 1 & 3 \\ 
      0 & 2 & 2 \\ 
      {0} & {1} & {2} \\   
\end{array}
\right)}_{C_\star}
\underbrace{\left( 
\begin{array}{ccc}
     1 & 0 & 0  \\
     0 & 1 & 0 \\ 
     0 & 0 & 1 
\end{array}
\right)}_{S_\star^\top}. 
\]
By Theorem~\ref{mainth}, the first column of $C_\star$ is uniquely identifiable, since the two conditions of Theorem~\ref{mainth} are satisfied: 
\begin{enumerate}

\item (Full-rank zero-region window--FRZRW) 
$C_\star(\mathcal{I},1) = 0$ for $\mathcal{I} = \{4,5\}$ while 
    \[
\rank \big( C_\star(\mathcal{I}, \mathcal{K}) \big) 
= \rank \left( 
\begin{array}{cc}
     2 & 2   \\
     1 & 2  
\end{array}
\right) 
= 2, \quad \text{ where } \mathcal{K} = \{2,3\}. 
\]  

    \item (Selective window) $S(1,:) = e_{(1)}^\top$ so that $R(:,1) = C_\star(:,1)$. 

\end{enumerate}

\end{example} 

\begin{remark}
In the example above, $S_\star$ is the identity matrix with $n=r$, which is  not realistic, is not a very interesting NMF decomposition (it is the trivial decomposition, $R = RI$), and would be useless in practice. 
However, for our purpose, such examples are enough. 
One could add any number of rows to $S_\star$ and replace the identity matrix by a diagonal matrix to make it more realistic, 
but it would not change our observations and discussions about the identifiability.
\end{remark}

\begin{remark}
The strengthen FRZRW condition compared to the zero-region window condition used in~\cite{lakeh2022predicting} comes from the fact that Theorem~\ref{mainth} provides a global uniqueness result. The result in~\cite{lakeh2022predicting} implicitly focuses on locally unique (a.k.a.\ locally rigid) solutions; see~\cite{krone2021uniqueness} for more details on local uniqueness and rigidity of Exact NMF solutions.   
\revise{For example, 
\begin{equation} \label{eq:coutnerex}
R =  
\underbrace{\left( 
\begin{array}{ccc}
      2 & 2 & 2   \\
      1 & 3 &  1 \\ 
      1 & 1 & 3 \\ 
      0 & 2 & 2 
\end{array}
\right)}_{=C_\star}
\underbrace{\left( 
\begin{array}{ccc}
     1 & 0 & 0  \\
     0 & 1 & 0 \\ 
     0 & 0 & 1 
\end{array}
\right)}_{=S_\star^\top}  
\end{equation} 
satisfies the zero-region window and selective window conditions 
for 
the first column of $C_\star$:  
the last row, $[0,2,2]$, is a zero-region window 
while 
the first row of $S_\star$ is $e_{(1)}\top$ and hence is 
a selective window.  
However, this first column is not uniquely identifiable, up to scaling, as there exists another decomposition where that column does not appear 
(up to scaling):   
\begin{equation} \label{eq:coutnerex2}
R = 
\left( 
\begin{array}{ccc}
      2 & 2 & 2   \\
      1 & 3 &  1 \\ 
      1 & 1 & 3 \\ 
      0 & 2 & 2 
\end{array}
\right)
= 
\left( 
\begin{array}{ccc}
      1 & 1 & 0   \\
      1 & 0 &  1 \\ 
      0 & 1 & 1 \\ 
      0 & 0 & 2 
\end{array}
\right)
\left( 
\begin{array}{ccc}
     1 & 2 & 0  \\
     1 & 0 & 2 \\ 
     0 & 1 & 1 
\end{array}
\right). 
\end{equation}  
However, in the first factorization, in~\eqref{eq:coutnerex}, the first column of $R$ is actually locally unique: any nearby Exact NMF factorization must contain $R(:,1)$ as a column, up to scaling. Note that, in the second factorization above, in~\eqref{eq:coutnerex2}, all columns of $R$ are locally partially identifiable; see Figure~\ref{counterex} for the corresponding NPP instance. 
\begin{figure}[ht!]
\begin{center}
\includegraphics[width=0.5\textwidth]{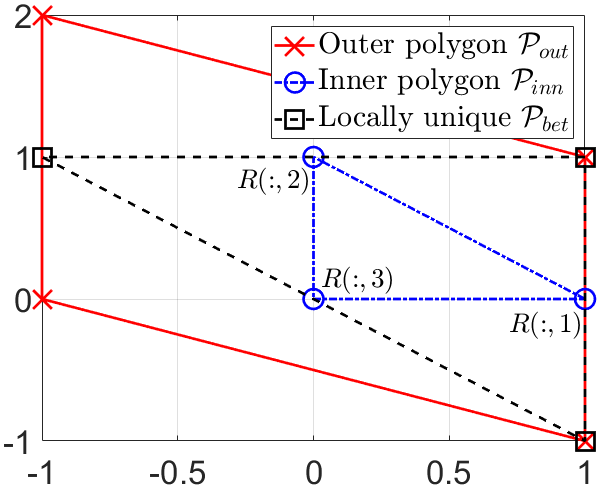} 
\caption{
Illustration of the NPP instance corresponding to the matrix $R$ in~\eqref{eq:coutnerex}. 
The nested polygon corresponding to the trivial factorization in~\eqref{eq:coutnerex}, $R= R\ I$,  is 
$\mathcal{P}_{inn}$ itself, 
while the nested polygon corresponding to the   factorization in~\eqref{eq:coutnerex2}
is denoted $\mathcal{P}_{bet}$.}.\label{counterex}
\end{center}
\end{figure} }
 An interesting direction of research would be to analyze conditions under which solutions are partially locally unique.  
\end{remark}

\begin{remark}[Full-rank zero-region window and sufficiently scattered condition] \label{rem:SSC}

It turns out the FRZRW condition for each column of $C$ is a necessary condition for the SSC. 
In fact, the SSC requires that $C$ has at least $r-1$ zero per column, 
while the submatrix $C(\mathcal{I},\mathcal{K})$ (using the same notation as in the proof of  Theorem~\ref{mainth}) needs to contain the all-one vector in its relative interior~\cite[Theorem 4.28]{gillis2020nonnegative} which requires that the rank of $C(\mathcal{I},\mathcal{K})$ is equal to $r-1$. 

Hence the SSC is stronger than the FRZRW condition. However, if both $C_\star$ and $S^\top$ satisfy the SSC, the Exact NMF of $R = CS^\top$ is unique, which is not the case for the FRZRW condition. For example, the following matrix 
\[
C_\star^\top = \left( 
\begin{array}{cccccc}
    1 & 0 & 0 & 1 & 1 & 0 \\ 
    0 & 1 & 0 & 1 & 0 & 1 \\ 
    0 & 0 & 1 & 0 & 1 & 1 \\ 
     1 & 1 & 1 & 0 & 0 & 0 \\
\end{array}
\right) 
\]
satisfies the FRZRW condition, but $CC^\top$ does not admit a unique NMF, e.g., $C_\star C_\star^\top = CC^\top$ where 
\[
C^\top =  \left( 
\begin{array}{cccccc}
    1    & 1   &  0  &    1 &    0  &   0\\ 
     1   &  0  &   1 &    0 &    1   &  0\\ 
     0   &  1  &   1 &    0  &   0  &   1\\  
     0   &  0   &  0 &    1 &    1  &   1
\end{array}
\right), 
\] 
see~\cite[Example 4.29]{gillis2020nonnegative}. 
\end{remark} 

In the next section,  Section~\ref{sub:geoThmain}, we show that Theorem~\ref{mainth} has a simple geometric interpretation in terms of the NPP. This will be used to obtain a new partial identifiability result for Exact NMF in Section~\ref{sec:mainth2} (Theorem~\ref{mainthgeo0}). 


\subsection{Geometric interpretation of  the restricted DBU theorem (Theorem~\ref{mainth})} \label{sub:geoThmain}

Let us consider an NPP with $\mathcal{P}_{inn} \subseteq \mathcal{P}_{out}$, and make the following simple observation.  
If a vertex, $v$, of the inner polytope, $\mathcal{P}_{inn}$, coincide with a vertex of the outer polytope, $\mathcal{P}_{out}$, then it has to belong to any nested polytope, $\mathcal{P}_{bet}$. In fact, since 
$\mathcal{P}_{inn} \subseteq \mathcal{P}_{bet} \subseteq \mathcal{P}_{out}$ and $v \in \mathcal{P}_{inn} \cap \mathcal{P}_{out}$, we must have $v \in \mathcal{P}_{bet}$. 

It turns out the conditions of Theorem~\ref{mainth} (namely, the selective window and FRZRW conditions) are equivalent to the condition that the inner and outer polytopes in the corresponding NPP share a vertex, but written in algebraic terms. Let us prove this equivalence. We will then use this geometric insight to provide a new partial identifiability result in Section~\ref{sec:mainth2}. 

In this section, we work on the outer polytope directly obtained from the reduction from Exact NMF to the NPP; see Section~\ref{sec:prelimgeo}. 
It is given by 
\begin{align*}
\mathcal{C} = \col(C) \cap \Delta & 
= \{ x \ | \  
x = Cz \geq 0, e^\top x = 1 \}  = \{ Cz \ | \  
Cz  \geq 0, e^\top z = 1 \}, 
\end{align*} 
where $C$ is normalized to be column stochastic, so that $x = Cz$ is column stochastic if and only if $e^\top z = 1$, since 
$e^\top x = e^\top C z = e^\top z$.  
It will be useful to note that the facets of $\mathcal{C}$ have the form 
$\{ Cz \in \mathcal{C} \ | \ (Cz)_i = C(i,:)z = 0\}$  for some~$i$.  

Let us define the smallest dimensional face 
of the outer polytope, $\mathcal{C}$, containing a given point. 
\begin{definition}[Minimal face of $\mathcal{C}$ containing $y$] Given a column stochastic matrix, $C \in \mathbb{R}_+^{m \times r}$, 
and the vector $y \in \mathcal{C} = \col(C) \cap \Delta$, we define 
\begin{align}
\mathcal{F}_C(y) 
& = 
\{ x \in \mathcal{C} \ | \ \supp(x) \subseteq \supp(y) \}.   \label{facety}
\end{align} 
\end{definition} 
The set $\mathcal{F}_C(y)$ can be characterized as follows 
\[
\mathcal{F}_C(y) 
= 
\{ Cz \ | \ 
 z \in \mathbb{R}^r, 
 Cz \geq 0,
(Cz)_i = 0 \text{ when } y_i = 0, 
z^\top e = 1 \}. 
\]  
This means that all the points in $\mathcal{F}_C(y)$ have to belong to the same facets of $\mathcal{C}$ as $y$. 
Hence $\mathcal{F}_C(y)$ is the face of $\mathcal{C}$ of minimal dimension containing $y$, because a face of a polytope is obtained by intersecting a subset of its facets.  
Note that a vertex of a poltyope is a 0-dimensional face, and 
hence $y$ is a vertex of $\mathcal{C}$ if and only if $\mathcal{F}_C(y) = \{ y \}$.  


Let us now prove that the FRZRW condition of Theorem~\ref{mainth} is equivalent to the fact that $C(:,k)$ is a vertex of $\mathcal{C} = \col(C) \cap \Delta$, that is, $\mathcal{F}_{C}\big( C(:,k) \big) = \big\{ C(:,k) \big\}$. 
Note that the selective window assumption of Theorem~\ref{mainth} will require that the inner polytope has a vertex corresponding to $C(:,k)$. 
\begin{lemma} \label{lemFRZRWvert} Given a nonsingular column stochastic matrix, $C \in \mathbb{R}_+^{m \times r}$, 
the FRZRW condition on the $k$th column of $C$ is equivalent 
to the following geometric condition 
\begin{equation} \label{geoFRZRW}   
\mathcal{F}_{C}\big( C(:,k) \big) 
\; = \; \big\{ C(:,k) \big\}. 
\end{equation}
\end{lemma} 
\begin{proof} Recall that $\mathcal{I}$ denotes the set of indices corresponding to the zero entries in $C(:,k)$, and let us denote $\xoverline[0.9]{\mathcal{I}}$ its complement which is the support of $C(:,k)$.  

$\boxed{\Rightarrow}$ Assume the FRZRW condition holds. Let $x = Cz \in \mathcal{F}_{C}\big( C(:,k) \big)$, that is, 
$Cz \geq 0$, 
$(Cz)_i = 0$ for $i \in \mathcal{I}$, 
$z^\top e = 1$. 
The condition $(Cz)_i = 0$ for $i \in \mathcal{I}$ can be written as  $C(\mathcal{I},:)z = 0$. Since $C(\mathcal{I},k) = 0$, by definition, this requires 
$C(\mathcal{I},\mathcal{K})z(\mathcal{K}) = 0$ where $\mathcal{K} = \{1,2,\dots,r\} \backslash \{k\}$ and the rank of $C(\mathcal{I},\mathcal{K})$ is $r-1$, by the FRZRW condition, and hence $z(\mathcal{K}) = 0$. This implies that $z = e_{(k)}$, since $z^\top e = 1$, and therefore~\eqref{geoFRZRW} holds. 
 
$\boxed{\Leftarrow}$ Assume~\eqref{geoFRZRW} holds. Since $C$ is nonsingular, \eqref{geoFRZRW} is equivalent to assuming that the solution to the system 
 \[
 Cz \geq 0, C(\mathcal{I},:) z = 0, e^\top z = 1, 
 \] 
 is unique and given by $z=e_{(k)}$. 
 The set of feasible solutions of the above system can be written as 
 \[
 \mathcal{Z} = \{ z \ | \ 
  C(\xoverline[0.9]{\mathcal{I}},:) z \geq 0, 
  C(\mathcal{I},\mathcal{K}) z(\mathcal{K}) = 0, 
  e^\top z = 1 \}. 
 \] 
Because of~\eqref{geoFRZRW} and $C$ being nonsingular, $\mathcal{Z} = \{e_{(k)}\}$.  
 Since $C(\xoverline[0.9]{\mathcal{I}},k) > 0$, by definition, $z = e_{(k)}$ belongs to the relative interior of $\mathcal{Z}$. 
Let us show that $\rank\big( C(\mathcal{I},\mathcal{K}) \big) < r-1$ implies that  the relative interior of $\mathcal{Z}$ is made of more than one point, leading to a contradiction, and hence  $\rank\big( C(\mathcal{I},\mathcal{K})\big) = r-1$ since $|\mathcal{K}| = r-1$.  
Let $y \neq 0$ belong to the kernel of $C(\mathcal{I},\mathcal{K})$, that is, $C(\mathcal{I},\mathcal{K}) y = 0$, 
 with $e^\top y = \beta \in \mathbb{R}$. 
 Let us define $z' \in \mathbb{R}^r$ as follows: $z'(\mathcal{K}) = \alpha y$ and $z'(k) = 1 - \alpha \beta$ so that $e^\top z' = 1$. For $\alpha$ sufficiently small, we have $C(\xoverline[0.9]{\mathcal{I}},:) z' > 0$, since 
 \[
 C(\xoverline[0.9]{\mathcal{I}},:) z' 
 = 
 \alpha C(\xoverline[0.9]{\mathcal{I}},\mathcal{K}) y 
 + 
  \underbrace{C(\xoverline[0.9]{\mathcal{I}},k)}_{> 0} (1 - \alpha \beta), 
 \] 
 and hence $z' \in \mathcal{Z}$ while $z' \neq e_{(k)}$. 
\end{proof}

Lemma~\ref{lemFRZRWvert} implies that, for $r=2$, the conditions of Theorem~\ref{mainth} are necessary and sufficient, since the  condition that $\mathcal{P}_{inn}$ and  $\mathcal{P}_{out}$ have a vertex that coincide is necessary and sufficient; see Section~\ref{sec:prev}.  


\subsection{New partial identifiability theorem for Exact NMF}  \label{sec:mainth2} 


By Theorem~\ref{th:nec}, for a column of $C$ to be identifiable, it has to belong to at least one facet of $\mathcal{C}$ where the other columns of $C$ are not located. In fact, its support cannot contain the support of any other  column of $C$. Geometrically, this means that, for $C(:,k)$ to be identifiable, a necessary condition is that $\mathcal{F}_{C}\big( C(:,k) \big)$ is a face of dimension smaller or equal than $r-2$ (recall that $\mathcal{C}$ has dimension $r-1$) where no other column of $C$ is located, that is, 
\[
C(:,j) \; \notin \;  \mathcal{F}_{C}\big( C(:,k) \big) \quad \text{ for all $j \neq k$}. 
\] 

Inspired by this observation, we 
obtain a new sufficient condition for partial identifiability in the following theorem. 
\begin{theorem} \label{mainthgeo0}
Let $R = C_\star S_\star^\top$ where 
$C_\star \in \mathbb{R}^{m \times r}_+$ and 
$S_\star \in \mathbb{R}^{r \times n}_+$ with $\rank(R) = r$. W.l.o.g., assume $R, C_\star$ and $S_\star^\top$ are column stochastic; see \eqref{sumtoonelemma} and \eqref{eq:sumtoone}.  
The $k$th column of $C_\star$ is identifiable if it satisfies the selective window condition, and 
there exists a subset, $\mathcal{J}$, of $r-1$ columns of $R$, namely $R(:,\mathcal{J})$, 
such that $\rank\left( R(:,\mathcal{J}) \right) = r-1$ and 
 for all $j \in \mathcal{J}$,  
\begin{equation} \label{eq:allijF}
\mathcal{F}_{C_\star}\big(C_\star(:,k)\big) \; \cap \;  
\mathcal{F}_{C_\star}\big(R(:,j)\big) \quad = \quad \emptyset , 
\end{equation}
that is, the minimal face on which the $k$th columns of $C_\star$ lies on does not intersect the minimal faces on which the columns of $R(:,\mathcal{J})$ lie on.   
\end{theorem}
\begin{proof}
Let $R = CS^\top$ be another exact NMF of $R$ of size $r$ where, w.l.o.g., we  assume $C$ and $S$ are column stochastic. 
Let $\mathcal{K} = \{1,2,\dots,r\} \backslash \{k\}$. 
We have 
\[
R = C_\star(:,k) S_\star(:,k)^\top +  C_\star(:,\mathcal{K}) S_\star(:,\mathcal{K})^\top
= \sum_{j=1}^r C(:,j) S(:,j)^\top. 
\] 
Let us introduce the following terminology: \revise{given two nonnegative matrices, $A$ and $B$, of the same dimension,  
we say that $A$ touches $B$ if there exists $(i,j)$ such that $A(i,j) > 0$ and $B(i,j) > 0$}.  
Below, we show that \eqref{eq:allijF} implies that, \revise{for  $j=1,2,\dots,r$, it is not possible that 
$C(:,j) S(:,j)^\top$ touches $C_\star(:,k) S_\star(:,k)^\top$ while $C(:,j) S(\mathcal{J},j)^\top$ touches   $R(:,\mathcal{J})$.} 
Since $R(:,\mathcal{J})$ has rank $r-1$, by the exclusion principle, exactly one rank-one factor touches $C_\star(:,k) S_\star(:,k)^\top$ and hence it has to coincide with it. 

Assume $C(:,p) S(:,p)^\top$ touches $C_\star(:,k) S_\star(:,k)^\top$ \revise{and 
$C(:,p) S(\mathcal{J},p)^\top$ touches} 
$R(:,\mathcal{J})$ for some $p$. 
As \mbox{$\rank\left(R(:,\mathcal{J})\right) = r-1$}, $R(:,j) \neq 0$ for all $j \in \mathcal{J}$.  
Since $C(:,p) S(:,p)^\top$ touches $C_\star(:,k) S_\star(:,k)^\top$, the support of $C(:,p)$ is contained in the support of $C_\star(:,k)$, and hence  
$C(:,p) \in \mathcal{F}_{C_\star}\big(C_\star(:,k)\big)$. 
By~\eqref{eq:allijF}, $C(:,p) \notin \mathcal{F}_{C_\star}\big(R(:,j)\big)$ for all $j \in \mathcal{J}$, that is, the support of $C(:,p)$ is not contained in the support of any column of $R(:,\mathcal{J})$ implying that it cannot touch \revise{any column of}  $R(:,\mathcal{J})$, a contradiction. 
\end{proof}

The condition in Theorem~\ref{mainthgeo0} implies that the $k$th column of $S_\star$ contains at least $r-1$ entries equal to zero, namely, $S_\star(\mathcal{J},k) = 0$,  since~\eqref{eq:allijF} requires that the support of $R(:,j)$ for $j\in \mathcal{J}$ does not contain the support of $C_\star(:,k)$. 

\begin{example} \label{examplesquare}
Let us consider the NPP where $\mathcal{P}_{out}$ is the unit square  $[0,1]^2$ as in Example~\ref{exNPP}, 
 while $\mathcal{P}_{inn}$ is the triangle with the vertices 
 $v_1=(0.5,0)$, $v_2=(0.2,1)$ and $v_3=(0.8,1)$; 
 see Figure~\ref{exNPPcorrec} for an illustration. 
 \begin{figure}[ht!]
\begin{center}
\includegraphics[width=0.8\textwidth]{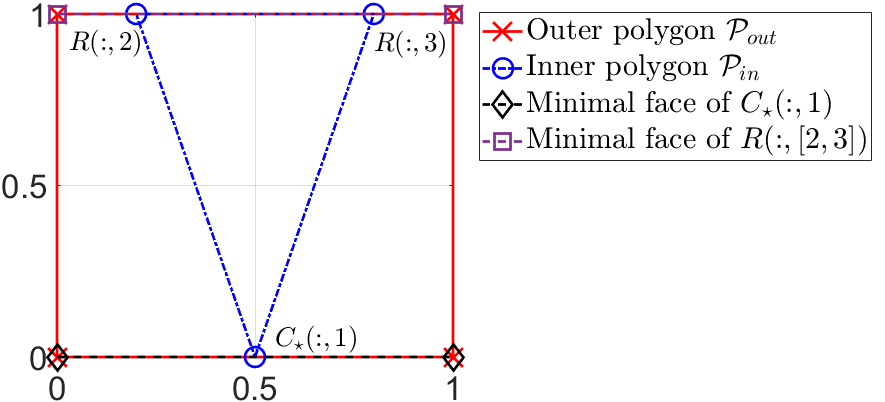} 
\caption{Illustration of the NPP instance described in Example~\ref{examplesquare}.\label{exNPPcorrec}} 
\end{center}
\end{figure}
 The matrix $R$ of the corresponding Exact NMF problem is given by $R(:,j) = F v_j + g$ for all $j$, that is, 
 \[
 R =  \left( \begin{array}{ccc}
      0  &   1 &   1  \\
      1  &   0 &   0  \\
      0.5  &   0.2 &   0.8  \\
      0.5  &   0.8 &   0.2 
 \end{array} \right) 
 = 
  \left( \begin{array}{ccc}
      0  &   1 &   1  \\
      1  &   0 &   0  \\
      0.5  &   0 &   1  \\
      0.5  &   1 &   0 
 \end{array} \right) 
 \left( \begin{array}{cccc}
      1  &   0 &   0  \\ 
      0  &   0.8 &   0.2  \\
      0  &   0.2 &   0.8  \\
 \end{array} \right). 
 \] 
 The first column of $C$ satisfies the selective window assumption, while  we observe on Figure~\ref{exNPPcorrec} that Theorem~\ref{mainthgeo0} applies using $R(:,[2,3])$ whose minimal faces do not intersect with that of $C_\star(:,1)$ 
 which is therefore identifiable. 
 Note that the restricted DBU Theorem~\ref{mainth} is not applicable to $C_\star(:,1)$ since it does not correspond to a vertex of $\mathcal{P}_{out}$. 
\end{example}

It is important to note that \begin{itemize}
    \item Theorem~\ref{mainthgeo0} does not subsume Theorem~\ref{mainth} which applies to a column of $C$ which is a vertex of $\mathcal{P}_{out}$ in which case the existence of a subset of columns of $R$ satisfying~\eqref{eq:allijF} is not necessary. 
For example, taking $\mathcal{P}_{out}$ as the square in two dimensions, as above, and taking the vertices of  $\mathcal{P}_{inn}$ as $v_1 = (0,0)$ (bottom left corner), $v_2 = (0,0.5)$ and $v_3 = (0.5,0)$, the conditions of Theorem~\ref{mainthgeo0} do not apply to the first column of $C$ (corresponding to $v_1$, since the minimal faces of $v_2$ and of $v_3$ contain $v_1$) while Theorem~\ref{mainth} does apply.  

\item For condition~\eqref{eq:allijF} to be satisfied, a necessary, but not sufficient, condition  is that the supports of $C_\star(:,k)$ and $R(:,j)$ are not contained in one another. In fact, this support condition implies that  
$\mathcal{F}_{C_\star}\big(C_\star(:,k)\big)$ and 
$\mathcal{F}_{C_\star}\big(R(:,j)\big)$ are distinct faces, but not that their intersection is empty.  
\end{itemize}

  Theorem~\ref{mainthgeo0} can be directly used to obtain a full identifiability result. 
\begin{corollary} \label{corr:mainthgeo0}
Let $R = C_\star S_\star^\top$ where 
$C_\star \in \mathbb{R}^{m \times r}_+$ and 
$S_\star \in \mathbb{R}^{r \times n}_+$ with $\rank(R) = r$. 
W.l.o.g., assume $R, C_\star$ and $S_\star^\top$ are column stochastic. If 
\begin{enumerate}
    \item Every column of $C_\star$ satisfies the selective window assumption, that is, $S_\star^\top$ is separable, and 
    
    \item The following holds for all $k \neq j$ 
\begin{equation*} 
\mathcal{F}_{C_\star}\big(C_\star(:,k)\big) \; \cap \;  
\mathcal{F}_{C_\star}\big(C_\star(:,j)\big) \quad = \quad \emptyset , 
\end{equation*}
\end{enumerate} 
then $(C_\star, S_\star)$ is (fully) identifiable. 
\end{corollary} 
\begin{proof}
 On one hand, Theorem~\ref{mainthgeo0} applies to all columns of $C_\star$, taking $R(:,\mathcal{J}) = C\left(:,\{1\dots,r\} \backslash \{k\}\right)$ 
  for all $k = 1,2,\dots,r$,  since $S_\star^\top$ is separable.  
  On the other hand $S_\star$ is identifiable since $C_\star$ is and $\rank(C_\star) = r$. 
\end{proof}
 For example, in two dimensions, when $r=3$,
full identifiability based on Corollary~\ref{corr:mainthgeo0} requires that, in the  NPP, the three vertices of $\mathcal{P}_{inn}$ corresponding to the three columns of $C_\star$ are located on three non-adjacent edges of the polygon $\mathcal{P}_{out}$. 
Note that this requires $\mathcal{P}_{out}$ to have at least six edges, that is, to be an $n$-gon with $n \geq 6$. This implies that $R$ needs to have at least 6 rows.

\begin{example} \label{ex:cube}
Let us take an example with  $r=4$, for which the NPP has dimension three.  
Consider the outer polytope as the unit cube in dimension 3, with $\mathcal{P}_{out} = [0,1]^3$, and take the vertices of $\mathcal{P}_{inn}$ as 
(0,0,0.75), (1, 0, 0.25), (0, 1, 0.25) and (1, 1, 0.75). 
This construction satisfies the conditions of Theorem~\ref{mainthgeo0} for all $k$  since the vertices of $\mathcal{P}_{inn}$ are on (minimal) faces (namely, edges) that do not intersect; see Figure~\ref{fig:cube} for an illustration. 
    \begin{figure}[ht!]
\begin{center}
\includegraphics[width=0.6\textwidth]{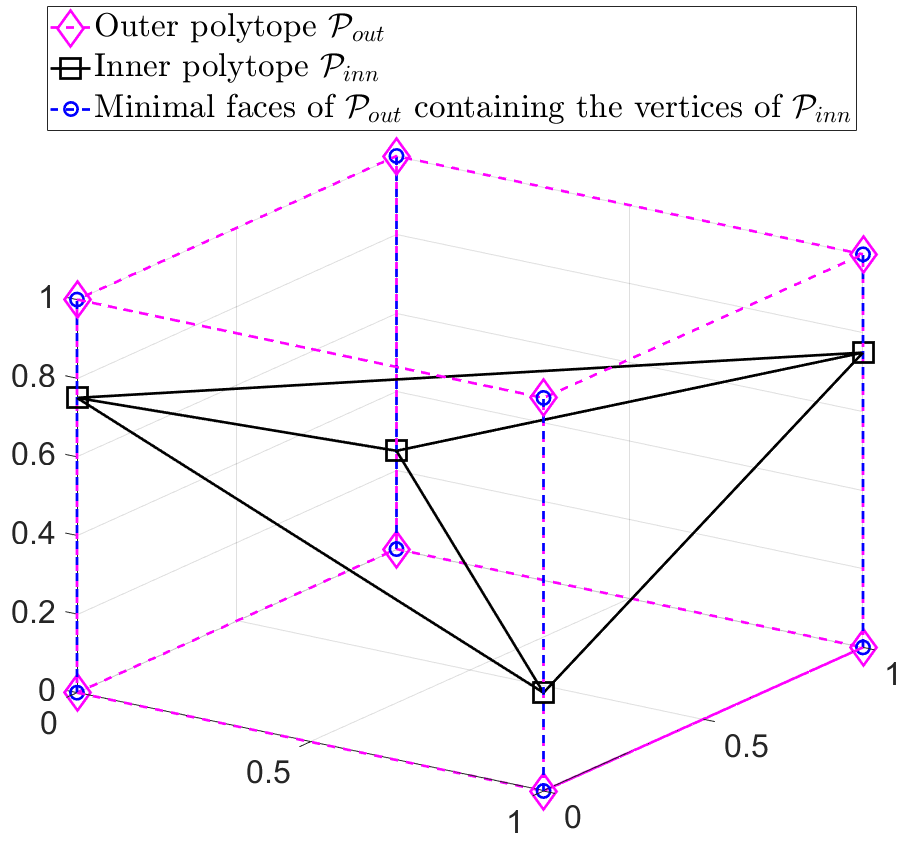} 
\caption{Geometric interpretation of Example~\ref{ex:cube} that satisfies the conditions of
Theorem~\ref{mainthgeo0}. \label{fig:cube}} 
\end{center}
\end{figure}

The corresponding $R$ is given by 
\[
 \begin{array}{c|cccc}
              &   (0,0,0.75) & (1, 0, 0.25) &  (0, 1, 0.25) &  (1, 1, 0.75) \\ 
 \hline  
x_1 \geq 0    &      0       &     1         &   0          &  1   \\
x_2 \geq 0    &      0       &     0         &   1          &  1 \\
x_3 \geq 0    &     0.75     &     0.25      & 0.25         &  0.75 \\
x_1 \leq 1    &     1        &     0         & 1            & 0 \\
x_2 \leq 1    &     1        &     1         & 0            & 0 \\
x_3 \leq 1    &     0.25     &    0.75       & 0.75         & 0.25
 \end{array}
\]
and therefore has a unique exact NMF, $R=RI$. Note that no column of $R$ satisfies the FRZRW condition of Theorem~\ref{mainth}, since $R$ has only two zero entries per column: Geometrically, no vertices of $\mathcal{P}_{inn}$ is a vertex of $\mathcal{P}_{out}$. 
\end{example}

\subsubsection{Is it easy to check the conditions of Theorem~\ref{mainthgeo0}?} 

Let $R = CS^\top$  be  an Exact NMF of size $r= \rank(R)$ where 
$R$ and $C$ are column stochastic (w.l.o.g.). 
A column of $R$, say the $j$th, fails to satisfy condition~\eqref{eq:allijF} 
if and only if 
there exists $x$ such that
\[
x \quad \in \quad  \mathcal{F}_{C_\star}\big(C_\star(:,k)\big) \; \cap \;  
\mathcal{F}_{C_\star}\big(R(:,j)\big). 
\] 
Such a $x$ exists if the following linear system in variable $z \in \mathbb{R}^r$ has a solution 
\[
x = Cz \geq 0, \; 
z^\top e = 1, \; 
(Cz)_i = 0 \text{ for all } 
i \in \mathcal{K}_{k,j} = \{ p \ | \ \revise{C(p,k)} = 0 \text{ or } \revise{R(p,j)} = 0 \}. 
\] 
This is a linear system in $r$ variables, with $\mathcal{O}(m)$ equalities and inequalities. 
In our implementation (see Section~\ref{sec:algo}), to avoid numerical issue, we rather solve the following linear optimization problem (which is always feasible): 
\begin{equation} \label{eq:linprog}
 \min_{z} \sum_{i \in \mathcal{K}_{k,j}} (Cz)_i 
 \quad \text{ such that }   \quad 
 Cz \geq 0 \text{ and } e^\top z = 1,  
\end{equation}
and check whether the optimal objective function value is below a given threshold (we used $10^{-6}$).

\subsection{Using partial identifiability theorems sequentially} \label{sec:recurDBU}

In this section, we provide a simple general framework to generalize partial identifiability theorems, assuming a subset of columns of $C_\star$ is already identifiable. 


\begin{theorem} \label{mainth2}
Let $R = C_\star S_\star^\top$ where 
$C_\star \in \mathbb{R}^{m \times r}_+$ and 
$S_\star \in \mathbb{R}^{r \times n}_+$ with $\rank(R) = r$. 
Assume $p$ columns of $C_\star$ are identifiable for \mbox{$p \in \{1,2,\dots,r-1\}$}, 
say the first $p$ w.l.o.g., that is, $C_\star(:,j)$ are identifiable for $j = 1,2,\dots,p$ (Definition~\ref{def:partialUniqNMF}).  
Let $\mathcal{J}$ be the index set corresponding to the columns of $R$ that do not contain the support of the first $p$ columns of $C_\star$. 

If \revise{$\rank\big( S_\star( \mathcal{J} ,  p+1:r ) \big) = r-p$}, 
and if the $(p+1)$th column of $C_\star$ can be certified to be identifiable in the Exact NMF $R(:, \mathcal{J})  
= C_\star(:,p+1:r) S_\star(\mathcal{J}, p+1:r)^\top$ of size $r-p$, then 
$C_\star(:,p+1)$ is identifiable in the Exact NMF of $R$ of size $r$. 
\end{theorem}
\begin{proof} 
Let $R = CS^\top$ be an Exact NMF of $X$ of size $r$ 
with  
$C \in \mathbb{R}^{m \times r}_+$ and 
$S \in \mathbb{R}^{n \times r}_+$. 
W.l.o.g., $C(:,1:p) = C_\star(:,1:p) D$ where $D$ is a diagonal matrix since the first $p$ columns of $C_\star$ are identifiable.  
We have 
\begin{equation} \label{decompXj2}
R(:, \mathcal{J})  
= 
C_\star S_\star(\mathcal{J},:)^\top  
= \sum_{q=p+1}^r C(:,q) S(\mathcal{J},q)^\top. 
\end{equation} 
The last equality follows by construction: the columns of $R(:,\mathcal{J})$ do not contain the support of the columns of 
$C_\star(:,1:p)$, which coincide with that of
$C(:,1:p)$, implying $S(\mathcal{J},q) = 0$ for all $q \leq p$. 
The fact that $\rank\big( S_\star( p+1:r , \mathcal{J}) \big) = r-p$ implies that $\rank\big(R(:,\mathcal{J})\big) = r-p$ since $\rank\big(C(:p+1:r)\big) = r-p$ as $\rank(C) = r$, and hence~\eqref{decompXj2} is an Exact NMF of rank $r-p$. By assumption, $C_\star(:,p+1)$ is identifiable in the Exact NMF~\eqref{decompXj2} so that one of the columns of $C(:,p+1:r)$ is equal to $C_\star(:,p+1)$, up to scaling. 
\end{proof} 


Let us illustrate Theorem~\ref{mainth2} on a simple example where all columns of $C_\star$ can be certified to be identifiable, using Theorem~\ref{mainth} sequentially.  
\begin{example} \label{ex:rank4th7} 
Let 
\[
R = \underbrace{\left( 
\begin{array}{cccc}
      0 & 1 & 1 & 1 \\
      0 & 1 & 2 & 3 \\ 
      0 & 1 & 2 & 1 \\ 
      1 & 0 & 1 & 2 \\ 
      1 & 0 & 2 & 1 \\
      1 & 1 & 0 & 1 \\
      1 & 1 & 1 & 0 \\
\end{array}
\right)}_{C_\star}
\underbrace{\left( 
\begin{array}{cccc}
     1 & 0 & 0 & 0  \\
     0 & 1 & 0 & 0 \\ 
     0 & 0 & 1 & 0 \\
      0 & 0 & 0 & 1
\end{array}
\right)}_{S_\star^\top} . 
\]
All columns of $C_\star$ are identifiable. The first one is by Theorem~\ref{mainth}. 
The second one is by combining Theorem~\ref{mainth2} and Theorem~\ref{mainth}: 
the last three columns of $R$ do not belong to the support of $C_\star(:,1)$, we have 
\[
R(:,2:4) = \underbrace{\left( 
\begin{array}{cccc}
       1 & 1 & 1 \\
       1 & 2 & 3 \\ 
       1 & 2 & 1 \\ 
       0 & 1 & 2 \\ 
       0 & 2 & 1 \\
       1 & 0 & 1 \\
       1 & 1 & 0 \\
\end{array}
\right)}_{C_\star(:,2:4)}
\underbrace{\left( 
\begin{array}{cccc}
       1 & 0 & 0 \\ 
       0 & 1 & 0 \\
        0 & 0 & 1
\end{array}
\right)}_{S_\star(2:4,2:4)^\top}, 
\]
where $\rank\big(S_\star(2:4,2:4)\big) = 3$. 
We can therefore apply Theorem~\ref{mainth} to the above Exact NMF of size $r-p=3$, which certifies the identifiability of $C_\star(:,2)$ (the selective window and FRZRW conditions hold). One can certify the identifiability of the last two columns of $C_\star$ in the same way. 
\end{example}

It is important to note that the conditions of Theorem~\ref{mainth2} do not necessarily become milder as $p$ increases. 
In practice, this means one needs to check $\sum_{p'=0}^p \binom{r}{p'}$ cases for each column of $C_\star$ not identified yet. 
However, this can be implemented relatively easily using recursion; 
see Section~\ref{sec:algo} for the details.  
Let us illustrate this on another example. 
\begin{example} \label{ex:rank4th7v2}
Let 
\[
R = \underbrace{\left( 
\begin{array}{cccc}
      0 & 1 & 1 & 1 \\
      0 & 1 & 2 & 3 \\ 
      0 & 1 & 2 & 1 \\ 
      1 & 0 & 1 & 2 \\ 
      1 & 0 & 2 & 1 \\
      1 & 0 & 0 & 1 \\
      1 & 1 & 1 & 0 \\
\end{array}
\right)}_{C_\star}
\underbrace{\left( 
\begin{array}{cccc}
     1 & 0 & 1 & 1  \\
     0 & 1 & 0 & 0 \\ 
     0 & 0 & 1 & 0 \\
     0 & 0 & 0 & 1
\end{array}
\right)}_{S_\star^\top} 
= 
\left( 
\begin{array}{cccc}
      0 & 1 & 1 & 1 \\
      0 & 1 & 2 & 3 \\ 
      0 & 1 & 2 & 1 \\ 
      1 & 0 & 2 & 3 \\ 
      1 & 0 & 3 & 2 \\
      1 & 0 & 1 & 2 \\
      1 & 1 & 2 & 1 \\
\end{array}
\right). 
\] 
As in Example~\ref{ex:rank4th7}, the first column of $C_\star$ is identifiable. Now, we realize that Theorem~\ref{mainth2} with $p=1$ for the second column is not applicable: the last two columns of $R$ do contain the support of $C_\star(:,1)$, so that $\mathcal{J} = \{2\}$, and 
$\rank\big( S(p+1:r, \mathcal{J} \big) = 1  < r-p-1 = 2$. 
However, the second column of $C_\star$ satisfies the conditions of 
Theorem~\ref{mainth}, and hence is identifiable. 

Note that the last two columns of $C_\star$ do not satisfy the selective window assumption, and it turns out that they are not identifiable; 
since another Exact NMF is given by $R = R I$.  
\end{example}

\subsection{Partial identifiability for Exact NMF when $r=3$} \label{sec:partialr3} 
We now analyze the case when $r=3$, which is of particular interest in the MCR literature, by providing a new condition for identifiability of two columns of $C_\star$. Before that, let us show the following lemma. 

\begin{lemma} \label{lem2C}
Let $R = C_\star S_\star^\top$ be an exact NMF of $R$ of size $r = \rank(R)$ where the $k$th column of $C_\star$ satisfies the selective window assumption.  
Let $R = CS^\top$ be an exact NMF of $R$ of size $r$. W.l.o.g., assume $C_\star$ and $C$ are column stochastic. 
If the $k$th column of $C_\star$ is not identified in $C$, that is, $C(:,j) \neq C_\star(:,k)$ for all $j$, then 
   there exists an index set $\mathcal{J}$ with $|\mathcal{J}| \geq 2$ such that 
\[
C(:,j) \in \mathcal{F}_{C_\star}\big( C_\star(:,k) \big) 
\; 
\text{ for } 
\; 
j \in \mathcal{J}. 
\]
\end{lemma}
\begin{proof}
Since $C_\star(:,k)$ satisfies the selective window assumption, that is, $C_\star(:,k) = \alpha R(:,j)$ for some~$j$ and $\alpha > 0$, we have $C_\star(:,k) = Cz$ for some $z \in \Delta$. 
The result then follows from the two observations: 
\begin{itemize}
    \item Since $C_\star(:,k) = Cz$, 
    $\supp\big(C(:,j)\big) \subseteq \supp\big( C_\star(:,k)\big)$ for all $j$ such that for $z_j > 0$. Therefore  
    $C(:,j) \in \mathcal{F}_{C_\star}\big( C_\star(:,k) \big)$ since $\col(C) = \col(C_\star) = \col(R)$.  
    
    \item Let $\mathcal{J} = \{ j \ | \ z_j > 0 \}$. 
    If $|\mathcal{J}| = 1$,  
    $C(:j) = C_\star(:,k)$ for some $j$, a contradiction, hence $|\mathcal{J}| \geq 2$.  
\end{itemize}
\end{proof}

\begin{theorem} \label{corthgeo}
Let $R = C_\star S_\star^\top$ where 
$C_\star \in \mathbb{R}^{m \times 3}_+$ and 
$S_\star \in \mathbb{R}^{3 \times n}_+$ with $\rank(R) = 3$, and $R$, $C_\star$ and $S_\star^\top$ normalized to be column stochastic as in~\eqref{eq:sumtoone}.   
Let us assume that two columns of $C_\star$  satisfy the selective window assumption, say the first and second one w.l.o.g. 
Let also the supports of $C_\star(:,1)$ and $C_\star(:,2)$ not be contained in one another. 
Then, these two columns are identifiable if 
 there exists a column of $R$, say the $j$th, such that: 
 
if $\mathcal{F}_{C_\star}\big(C_\star(:,2)\big) \cap 
\mathcal{F}_{C_\star}\big(C_\star(:,1)\big) = \emptyset$, 
\begin{equation} \label{lem3eq1}
R(:,j) \quad \notin \quad \conv\Big(\big[ C_\star(:,1), \mathcal{F}_{C_\star}\big(C_\star(:,2)\big) \big]\Big) \cup  \conv\Big(\big[ C_\star(:,2), \mathcal{F}_{C_\star}\big(C_\star(:,1)\big) \big]\Big), 
\end{equation} 

else 
\begin{equation} \label{lem3eq2} 
R(:,j) \quad \notin \quad \conv\Big( \mathcal{F}_{C_\star}\big(C_\star(:,1)\big), \mathcal{F}_{C_\star}\big(C_\star(:,2)\big) \Big). 
\end{equation} 

\end{theorem}
\begin{proof} 
  Let $R = CS^\top$ be an exact NMF of $R$ of size $r=3$. 
 The proof mostly relies on Lemma~\ref{lem2C}: if $C_\star(:,k)$ is not identified, then there are least two columns of $C$ in $\mathcal{F}_{C_\star}\big(C_\star(:,k)\big)$. 
Note that, for $r=3$, $\mathcal{C}$ is a polygon, and hence there are three types of facets depending on their dimension: 
the 2-dimensional polygon itself, $\mathcal{C}$, 
1-dimensional segments, and 0-dimensional vertices. 
By~\eqref{lem3eq1} or~\eqref{lem3eq2}, 
$\mathcal{F}_{C_\star}\big(C_\star(:,j)\big)$ for $j=1,2$ cannot be the polygon itself and hence are either segments or vertices. 

  \textbf{Case 1}: $\mathcal{F}_{C_\star}\big(C_\star(:,2)\big) \cap 
\mathcal{F}_{C_\star}\big(C_\star(:,1)\big) = \emptyset$. 
Since $C$ has three columns, there cannot be four columns of $C$ in $\mathcal{F}_{C_\star}\big(C_\star(:,k)\big)$ for $k \in \{1,2\}$ and therefore   $C_\star(:,k)$ is identified for $k=1$ or $k=2$, say $C_\star(:,1)$ w.l.o.g. 
Then, because of~\eqref{lem3eq1}, $C_\star(:,2)$ must also be identified otherwise $R(:,j)$ cannot be reconstructed. In fact, if $C_\star(:,2)$ was not identified,  the two columns of $C$ not multiple of $C_\star(:,1)$ (which is identified) must be on $\mathcal{F}_{C_\star}\big(C_\star(:,2)\big)$, a contradiction between the fact that $R(:,j) = CS(j,:)^\top$ and Equation~\eqref{lem3eq1}.  

\textbf{Case 2}: $\mathcal{F}_{C_\star}\big(C_\star(:,2)\big) \cap 
\mathcal{F}_{C_\star}\big(C_\star(:,1)\big) \neq \emptyset$. 
The two facets $\mathcal{F}_{C_\star}\big(C_\star(:,1)\big)$ and $\mathcal{F}_{C_\star}\big(C_\star(:,2)\big)$ intersect in a vertex. In fact, for $r=3$, $\mathcal{F}_{C_\star}\big(C_\star(:,1)\big)$ and $\mathcal{F}_{C_\star}\big(C_\star(:,2)\big)$ are adjacent segments of $\mathcal{C}$ since the support of $C_\star(:,1)$ does not contain and is not contained in that of $C_\star(:,2)$. Moreover, by the same support condition, 
$C_\star(:,1) \notin \mathcal{F}_{C_\star}\big(C_\star(:,2)\big)$, and vice versa.  Therefore, if $C_\star(:,1)$ or $C_\star(:,2)$ is not identified, the three columns of $C$ belong to $\mathcal{F}_{C_\star}\big(C_\star(:,1)\big) \cup  
\mathcal{F}_{C_\star}\big(C_\star(:,2)\big)$, which is a contradiction since $R(:,j)$ does not belong to the convex hull of these sets, see~\eqref{lem3eq2}, 
and hence $C$ cannot be used to reconstruct $R(:,j)$. 
\end{proof}

\begin{example} Let us construct two examples to illustrate the two cases in Theorem~\ref{corthgeo}. 
To do so, we use the equivalence of Exact NMF with the NPP, and use the same outer polygon $\mathcal{P}_{out} = [0,1]^2$ as in Example~\ref{exNPP}. 

In the first case of Theorem~\ref{corthgeo}, the two minimal faces containing $C_\star(:,1)$ and $C_\star(:,2)$ do not intersect. 
For example, one can take the two points (0.5,0) and (0.5,1); see Figure~\ref{fig:ex3}.
    \begin{figure}[ht!]
\begin{center}
\includegraphics[width=0.7\textwidth]{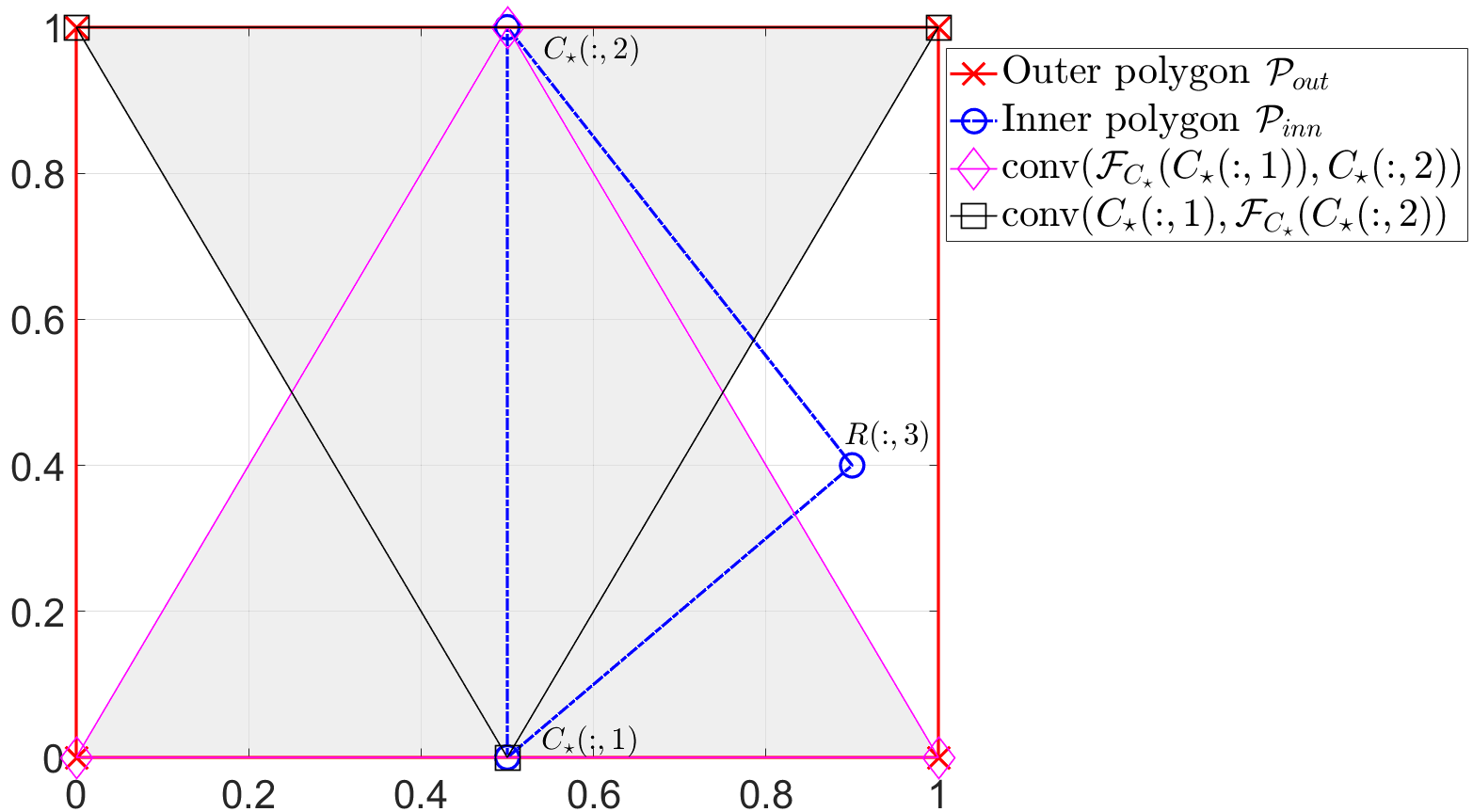} 
\caption{Geometric interpretation of the Exact NMF problem: identifiability of the first two columns of $C_\star$, case 1 of Theorem~\ref{corthgeo}. \label{fig:ex3}} 
\end{center}
\end{figure} 
These two points correspond to 
\[
C_\star(:,1) =  F (0.5,0) + g = (0,1,0.5,0.5)^\top, 
\text{ and } 
C_\star(:,2) =  F (0.5,1) + g = (1,0,0.5,0.5)^\top. 
\]
If a column of $R = C_\star S_\star^\top$ does not belong to 
\[ 
\conv\Big(C_\star(:,1), \mathcal{F}_{C_\star}\big(C_\star(:,2)\big)\Big) \cup  
\conv\Big(C_\star(:,2), \mathcal{F}_{C_\star}\big(C_\star(:,1)\big)\Big), 
\] 
then both columns are identifiable.  
This is the case on Figure~\ref{fig:ex3} with  
\[
R(:,3) =  F \revise{(0.9,0.4)} + g = (0.4,0.6,0.9,0.1)^\top. 
\]

In the second case of Theorem~\ref{corthgeo}, the two minimal faces containing $C_\star(:,1)$ and $C_\star(:,2)$ do intersect. 
For example, one can take the two points (0.5,0) and (0,0.5); see Figure~\ref{fig:ex2}.
These two points correspond to 
\[
C_\star(:,1) =  F (0.5,0) + g = (0,1,0.5,0.5)^\top, 
\text{ and } 
C_\star(:,2) =  F (0,0.5) + g = (0.5,0.5,0,1)^\top. 
\]
If a column of $R = C_\star S_\star^\top$ does not belong to 
\[
\conv\Big( 
\mathcal{F}_{C_\star}\big(C_\star(:,1)\big), 
\mathcal{F}_{C_\star}\big(C_\star(:,2)\big) 
\Big), 
\] 
then both columns are identifiable. 
    \begin{figure}[ht!]
\begin{center}
\includegraphics[width=0.45\textwidth]{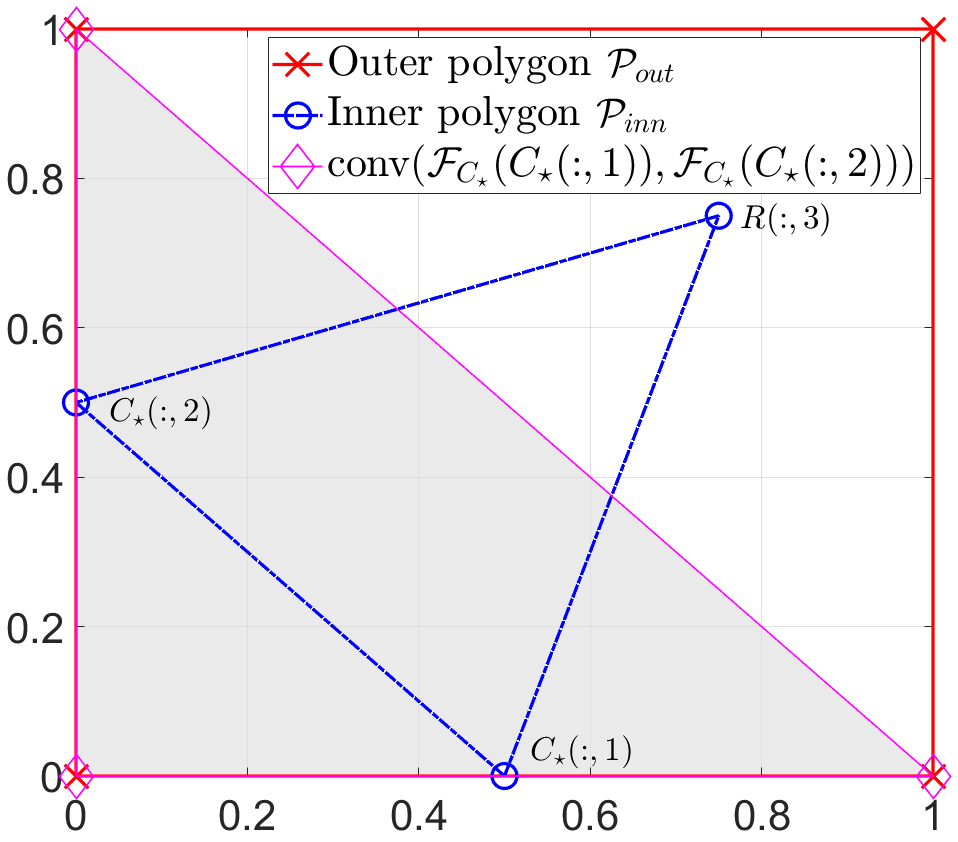} 
\caption{Geometric interpretation of the Exact NMF problem: identifiability of the first two columns of $C_\star$, case 2 of Theorem~\ref{corthgeo}. \label{fig:ex2}} 
\end{center}
\end{figure}
This is the case on Figure~\ref{fig:ex2}, with  
\[
R(:,3) =  F (0.75,0.75) + g = (0.75,0.25,0.75,0.25)^\top. 
\] 
\end{example}

\section{Applications of the new partial identifiability results} \label{sec:appli}

In this section, we first discuss whether the conditions of our identifiability results are reasonable in practice. 
Then we propose an algorithm, Algorithm~\ref{algo1}, that combines our partial identifiability results to certify the partial identifiability results for a given input matrix $R$. Finally, we illustrate its use on an example  from the chemometrics literature.

\subsection{Are the conditions of our identifiability results reasonable?}

All our proposed identifiability results rely on two facts: 
\begin{enumerate}
    \item Some columns of $C$ satisfy the selective window assumption, this requires some rows of $S$ to be unit vectors (up to scaling).  
    
    \item Some degree of sparsity of $C$. (Note that this is a necessary condition for identifiability of Exact NMF; see Theorem~\ref{th:necpartial}).  
\end{enumerate} 
This implies some degree of sparsity in $R=CS^\top$, since some columns of $R$ will be equal to the columns of $C$ that have zero entries.  

The selective window assumption is reasonable in many applications; see, e.g., the discussion in~\cite[Chapter 7]{gillis2020nonnegative} about separability and the references therein. 
However, sparsity is not necessarily natural in all applications where separability arises, e.g., in blind hyperspectral unmixing where spectral signatures are typically dense, and in facial feature extraction where facial images are dense. 
However, it is reasonable in other applications. For example,  
\begin{itemize}
    \item MCR: 
    The spectral content of some sources/components can be high (overlapping) while it is zero/small for others at some wavelength (selective window assumption). Moreover, all components are not present at all time window (sparsity); see an example in Section~\ref{sec:numexp}. 
    
    \item Topic modeling: the presence of anchor words, which are words associated to a single topic, is a reasonable assumption~\cite{arora2013practical} (selective window), while most documents only discuss a few topics (sparsity).  
    For example, using the widely data set tdt2\_top30 
    (9394 documents and 19528 words) we computed an approximate Exact NMF of the form $R \approx \tilde{R} = CS^\top$ for $r \in \{1,2,\dots,100\}$ 
     using a separable NMF algorithm, namely the successive projection algorithm (SPA)~\cite{Araujo01}, one of the most widely used ones.  
     All decompositions $\tilde{R} = CS^\top$ obtained are certified to be unique using the restricted DBU theorem (Theorem~\ref{mainth}). (Here we can only certify that the Exact NMF of the approximation is identifiable, since there does not exist an Exact NMF of $R$ for a small~$r$; in fact\footnote{We stopped the modified Gram-Schmidt with column pivoting at $r=800$, after about 5 hours on a standard laptop.}, $\rank(R) \geq 800$. This is often the case in practice because of the noise and model misfit.)  
    
\end{itemize} 

In summary, our results will likely 
apply when $R$ contains some columns with sufficiently many zero entries, while the selective window assumption makes sense.

\subsection{An algorithm to check partial identifiability} \label{sec:algo}

Relying on our new theoretical results, we provide in this section an algorithm that provides partial identifiability guarantees for the Exact NMF of a given nonnegative matrix $R$; see Algorithm~\ref{algo1}.  
As for all the results of this paper, 
Algorithm~\ref{algo1} assumes $\rank(R) = \rank_+(R)$ which is reasonable in most real-world applications.  Algorithm~\ref{algo1} is available from \url{https://bit.ly/partialNMFidentifv2}, along with all the examples presented in the paper (and two other ones).

\algsetup{indent=2em}
\begin{algorithm}[ht!]
\caption{Partial identifiability guarantees for $C$ in an Exact NMF $R=CS^\top$ of size $\rank(R)$} \label{algo1}
\begin{algorithmic}[1] 
\REQUIRE 
An Exact NMF of $R = CS^\top$, with $C \in \mathbb{R}^{m \times r}_+$ and $S \in \mathbb{R}^{n \times r}_+$ where $r = \rank(R)$.  



\ENSURE A subset $\mathcal{K}$ of the columns of $C$ that are guaranteed to be identifiable (Definition~\ref{def:partialUniqNMF}).   

    \medskip  













\STATE  Normalize $(R,C,S^\top)$ so that they are column stochastic; see~\eqref{eq:sumtoone}. 

\STATE Initialize $\mathcal{K} = \emptyset$. 

\STATE Let $\mathcal{L}$ be the set of columns of $C$ that satisfy the selective window assumption, that is, 

$\mathcal{L} = \{ i \ | \  \text{there exist } k \text{ and } \alpha > 0 
\text{ such that } S(k,:) = \alpha e_{(i)}^\top \}$. 


\STATE \emph{\% Use Theorems~\ref{mainth} and~\ref{mainthgeo0}}  


\FOR{every index in $k \in \mathcal{L} \backslash \mathcal{K}$}

\IF{$\rank\big(C(\mathcal{I},:)\big)  = r-1$ where $\mathcal{I} = \{ i \ | \ C(i,k) = 0\}$}

\STATE  $\mathcal{K} \leftarrow \mathcal{K} \cup \{k\}$. 

\ENDIF 

\IF{there exists $\mathcal{J}$ s.t.\  
$\mathcal{F}_C\big(C(:,k)\big) \cap \mathcal{F}_C\big(R(:,j)\big) \revise{ = \emptyset}$ for all $j \in \mathcal{J}$, $\rank\left( R(:,\mathcal{J}) \right) = r-1$ }

\STATE  $\mathcal{K} \leftarrow \mathcal{K} \cup \{k\}$. 

\ENDIF

\ENDFOR

\STATE \emph{\% Use Theorem~\ref{mainth2} combined Theorems~\ref{mainth} and~\ref{mainthgeo0}, recursively} 

\STATE $i = 1$ 

\WHILE{$i \leq |\mathcal{K}|$} 

    \STATE $\mathcal{P} = \{1,2,\dots,r\} \backslash \{\mathcal{K}(i)\}$ 
    \qquad \emph{\%  $\mathcal{K}(i)$ is the $i$th element in the set $\mathcal{K}$} 
    
    \STATE Let $\mathcal{J}$ be the subset of columns of $R$ not containing the support of $C\big(:,\mathcal{K}(i)\big)$. 
    
    \IF{$\rank\big(S(\mathcal{J},\mathcal{P})\big) = r-1$}
    
        \STATE $\mathcal{K}'$ $=$ Algorithm~\ref{algo1}$\big(C(:,\mathcal{P}), S(\mathcal{J},\mathcal{P})\big)$ 
        
        \STATE $\mathcal{K} \leftarrow \mathcal{K} \cup \mathcal{P}(\mathcal{K}')$, 
        
    \ENDIF
    
    \STATE $i \leftarrow i+1$ 
    
\ENDWHILE

\STATE \textbf{if} $r=3$ \textbf{then} use Theorem~\ref{corthgeo} for pairs of indices in $\mathcal{K}$.     





%

\end{algorithmic}  
\end{algorithm}


\begin{remark}[Use of Algorithm~\ref{algo1} for real-world data] 
NMF algorithms may return $C$ and $S$ with many entries close to zero but not exactly zero (e.g., if the algorithm has not converged). 
Therefore, to check whether your computed solution is close to being (partially) identifiable, you can set these entries to zero using some threshold strategy, and then call Algorithm~\ref{algo1}. 

\revise{Another strategy is to weaken the sharp zero condition in the sense of generalized Borgen plots~\cite{jurss2015generalized}. 
This strategy is useful for experimental noisy data which may include also a background subtraction resulting in small negative entries.} 
\end{remark}

\subsection{Numerical example from the chemometrics literature} \label{sec:numexp}

Let us consider the 5-component data set from~\cite[Section 4.3]{lakeh2022predicting}; see Figure~\ref{fig:5comp}. 
\begin{figure}[ht!]
\begin{center}
\begin{tabular}{cc}
  \includegraphics[width=0.48\textwidth]{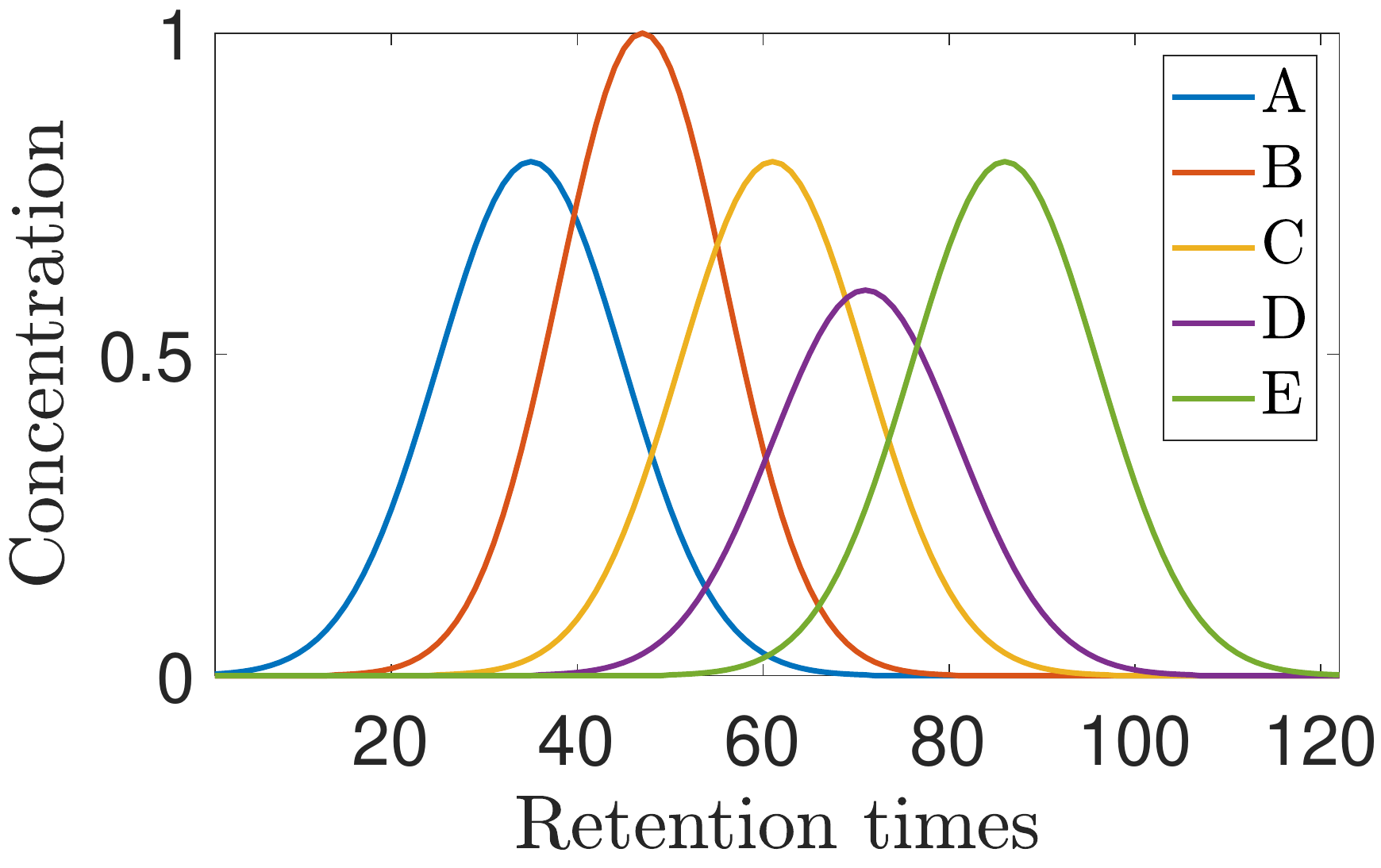} 
   &  \includegraphics[width=0.48\textwidth]{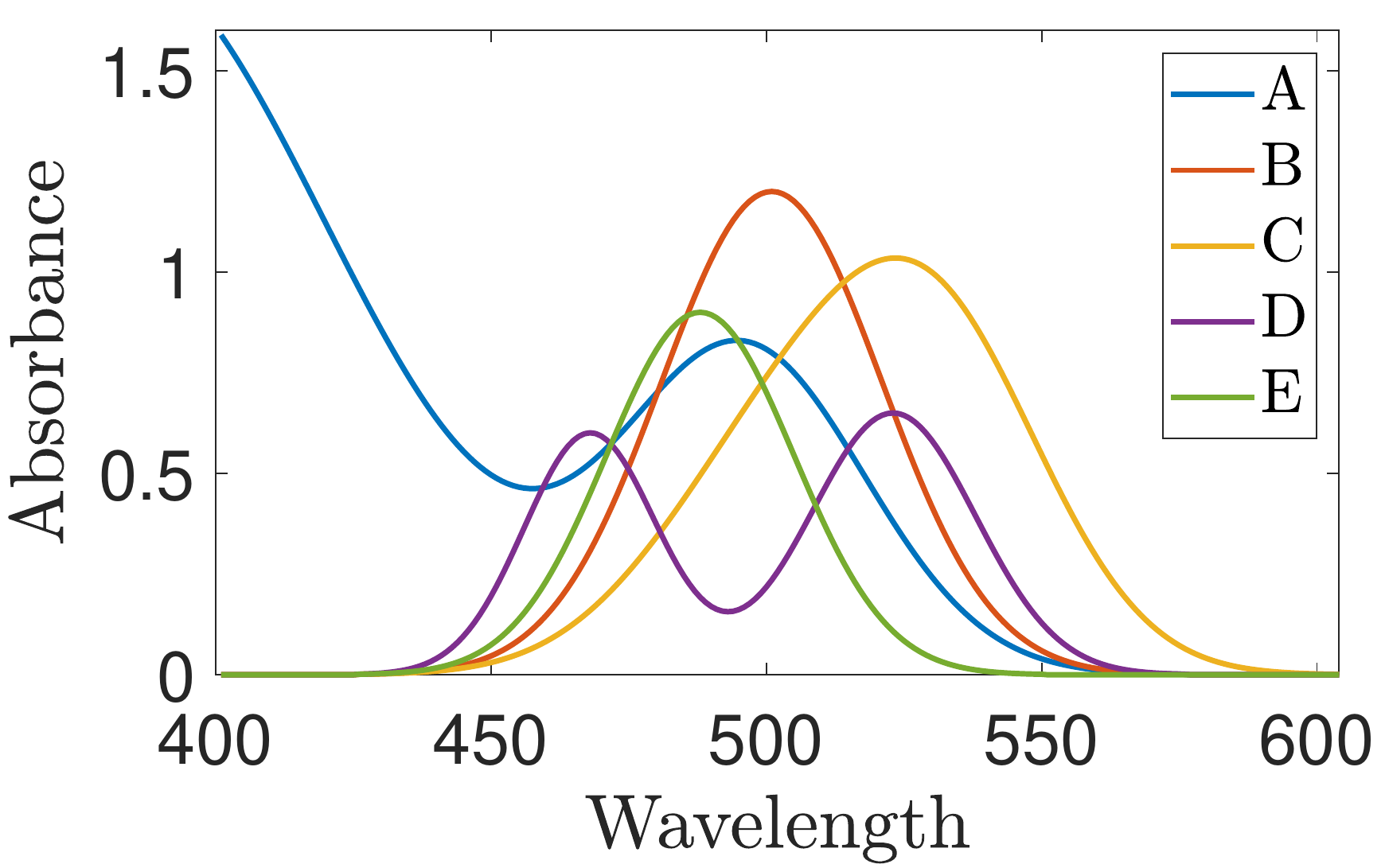} 
\end{tabular}
\caption{A five-component data set. 
On the left, the elution profiles of the chemical components which are the columns of $C^\star$. 
On the right, the spectra of the chemical components which are the columns of $S^\star$. 
\label{fig:5comp}} 
\end{center}
\end{figure} 
Algorithm~\ref{algo1} certifies that the first and third columns of $C^\star$ are identifiable, 
and the fifth column of $S^\star$: 

 \texttt{[K,L] = check\_partial\_identif(C,S)} 

 \texttt{K = [1 3], L = 5.} 

For example, for $C^\star$, using the restricted DBU theorem, the elution profile (that is, the columns of $C^\star$) that can be guaranteed to be identifiable are the first (A) and third ones (C): they satisfy the selective window condition (first wavelengths for A, last ones for C) while the FRZRW condition can be checked (the other elution profiles have rank $r-1$ when restricted to the entries where the corresponding column of $C^\star$ is zero). 
Note that the first columns of $S^\star$ satisfies the selective window condition, but not the FRZRW condition because when its spectrum is equal to zero, the spectrum of other components also are (namely, all of them but C). 
These are the same conclusion as in~\cite{lakeh2022predicting}.  \\




\section{Conclusion}

In this paper, we have provided the following partial identifiability  results for Exact NMF: 
\begin{itemize}
    \item a rigorous description and proof of the restricted DBU theorem (Theorem~\ref{mainth}). 
    
    \item a new partial identifiability result  based on the geometric interpretation of the restricted DBU theorem (Theorem~\ref{mainthgeo0}). 
    
    \item a sequential approach to guarantee the identifiability of more factors (Theorem~\ref{mainth2}). 
    
\end{itemize}

Since this paper is, to the best of our knowledge, the first to rigorously investigate partial identifiability of Exact NMF, there is still a lot to be done. 
In particular, can stronger partial identifiability theorems be obtained? For example, is it possible to provide partial identifiability results for several components simultaneously under weaker conditions? We have done this for the case $r=3$ considering two components at a time (see Theorem~\ref{corthgeo}), and this idea can probably be generalized to larger~$r$. 
In particular, considering all factors allows one to relax the selective window assumption (a.k.a.\ separability, which is rather strong) to the sufficiently scattered condition; see Theorem~\ref{th:ssc}.

\revise{
\section*{Acknowledgement} 
We thank the reviewers for their insightful comments that helped us improve the paper. 
}

\small 

\bibliographystyle{spmpsci}
\bibliography{partial_unique_NMF}

\begin{thebibliography}{10}
\providecommand{\url}[1]{{#1}}
\providecommand{\urlprefix}{URL }
\expandafter\ifx\csname urlstyle\endcsname\relax
  \providecommand{\doi}[1]{DOI~\discretionary{}{}{}#1}\else
  \providecommand{\doi}{DOI~\discretionary{}{}{}\begingroup
  \urlstyle{rm}\Url}\fi

\bibitem{abdolali2021simplex}
Abdolali, M., Gillis, N.: Simplex-structured matrix factorization:
  Sparsity-based identifiability and provably correct algorithms.
\newblock SIAM Journal on Mathematics of Data Science \textbf{3}(2), 593--623
  (2021)

\bibitem{andersons2022analytical}
Andersons, T., Sawall, M., Neymeyr, K.: Analytical enclosure of the set of
  solutions of the three-species multivariate curve resolution problem.
\newblock Journal of Mathematical Chemistry \textbf{60}, 1750--1780 (2022)

\bibitem{Araujo01}
Ara\'ujo, U., Saldanha, B., Galv\~ao, R., Yoneyama, T., Chame, H., Visani, V.:
  The successive projections algorithm for variable selection in spectroscopic
  multicomponent analysis.
\newblock Chemometrics and Intelligent Laboratory Systems \textbf{57}(2),
  65--73 (2001)

\bibitem{arora2013practical}
Arora, S., Ge, R., Halpern, Y., Mimno, D., Moitra, A., Sontag, D., Wu, Y., Zhu,
  M.: A practical algorithm for topic modeling with provable guarantees.
\newblock In: International conference on machine learning, pp. 280--288. PMLR
  (2013)

\bibitem{Beyramysoltan2014BPduality}
Beyramysoltan, S., Abdollahi, H., Rajk{\'o}, R.: Newer developments on
  self-modeling curve resolution implementing equality and unimodality
  constraints.
\newblock Analytica Chimica Acta \textbf{827}, 1--14 (2014)

\bibitem{Borgen1985compres3comp}
Borgen, O.S., Kowalski, B.R.: An extension of the multivariate
  component-resolution method to three components.
\newblock Analytica Chimica Acta \textbf{174}, 1--26 (1985)

\bibitem{brown2020comprchemom}
Brown, S., Tauler, R., Walczak, B. (eds.): Comprehensive Chemometrics Chemical
  and Biochemical Data Analysis, 2nd Edition.
\newblock Elsevier (2020)

\bibitem{cichocki2009nonnegative}
Cichocki, A., Zdunek, R., Phan, A.H., Amari, S.i.: Nonnegative Matrix and
  Tensor Factorizations: Applications to Exploratory Multi-Way Data Analysis
  and Blind Source Separation.
\newblock John Wiley \& Sons (2009)

\bibitem{common2010blindsource}
Common, P., Jutten, C.: Handbook of Blind Source Separation: Independent
  Component Analysis and Applications.
\newblock Academic Press (2010)

\bibitem{donoho2004does}
Donoho, D., Stodden, V.: When does non-negative matrix factorization give a
  correct decomposition into parts?
\newblock In: Advances in Neural Information Processing Systems (NIPS), pp.
  1141--1148 (2004)

\bibitem{xiao2019uniq}
Fu, X., Huang, K., Sidiropoulos, N.D., Ma, W.K.: Nonnegative matrix
  factorization for signal and data analytics: Identifiability, algorithms, and
  applications.
\newblock IEEE Signal Processing Magazine \textbf{36}(2), 59--80 (2019)

\bibitem{gillis2020nonnegative}
Gillis, N.: Nonnegative Matrix Factorization.
\newblock SIAM, Philadelphia (2020)

\bibitem{Golshan2016review3compRFS}
Golshan, A., Abdollahi, H., Beyramysoltan, S., Maeder, M., Neymeyr, K.,
  Rajk{\'o}, R., Sawall, M., Tauler, R.: A review of recent methods for the
  determination of ranges of feasible solutions resulting from soft modelling
  analyses of multivariate data.
\newblock Analytica Chimica Acta \textbf{911}, 1--13 (2016)

\bibitem{Golshan2011grid3}
Golshan, A., Abdollahi, H., Maeder, M.: Resolution of rotational ambiguity for
  three-component systems.
\newblock Analytical Chemistry \textbf{83}(3), 836--841 (2011)

\bibitem{Golshan2013}
Golshan, A., Maeder, M., Abdollahi, H.: Determination and visualization of
  rotational ambiguity in four-component systems.
\newblock Analytica Chimica Acta \textbf{796}, 20--26 (2013)

\bibitem{GRANDE2000}
Grande, B.V., Manne, R.: Use of convexity for finding pure variables in two-way
  data from mixtures.
\newblock Chemometrics and Intelligent Laboratory Systems \textbf{50}(1),
  19--33 (2000)

\bibitem{huang2013non}
Huang, K., Sidiropoulos, N.D., Swami, A.: Non-negative matrix factorization
  revisited: Uniqueness and algorithm for symmetric decomposition.
\newblock IEEE Transactions on Signal Processing \textbf{62}(1), 211--224
  (2013)

\bibitem{jurss2015generalized}
J{\"u}r{\ss}, A., Sawall, M., Neymeyr, K.: On generalized {Borgen} plots. i:
  From convex to affine combinations and applications to spectral data.
\newblock Journal of Chemometrics \textbf{29}(7), 420--433 (2015)

\bibitem{krone2021uniqueness}
Krone, R., Kubjas, K.: Uniqueness of nonnegative matrix factorizations by
  rigidity theory.
\newblock SIAM Journal on Matrix Analysis and Applications \textbf{42}(1),
  134--164 (2021)

\bibitem{lakeh2022predicting}
Lakeh, M.A., Abdollahi, H., Rajk{\'o}, R.: Predicting the uniqueness of single
  non-negative profiles estimated by multivariate curve resolution methods.
\newblock Analytica Chimica Acta \textbf{1199}, 339,575 (2022)

\bibitem{laurberg2008theorems}
Laurberg, H., Christensen, M.G., Plumbley, M.D., Hansen, L.K., Jensen, S.H.:
  Theorems on positive data: On the uniqueness of {NMF}.
\newblock Computational Intelligence and Neuroscience \textbf{2008} (2008)

\bibitem{laursen2022sampling}
Laursen, R., Hobolth, A.: A sampling algorithm to compute the set of feasible
  solutions for non-negative matrix factorization with an arbitrary rank.
\newblock SIAM Journal on Matrix Analysis and Applications \textbf{43}(1)
  (2022)

\bibitem{lawton1971self}
Lawton, W.H., Sylvestre, E.A.: Self modeling curve resolution.
\newblock Technometrics \textbf{13}(3), 617--633 (1971)

\bibitem{leplat2021multiplicative}
Leplat, V., Gillis, N., Idier, J.: Multiplicative updates for nmf with
  $\beta$-divergences under disjoint equality constraints.
\newblock SIAM Journal on Matrix Analysis and Applications \textbf{42}(2),
  730--752 (2021)

\bibitem{maeder1987evolving}
Maeder, M.: Evolving factor analysis for the resolution of overlapping
  chromatographic peaks.
\newblock Analytical Chemistry \textbf{59}(3), 527--530 (1987)

\bibitem{malinowski1992window}
Malinowski, E.R.: Window factor analysis: Theoretical derivation and
  application to flow injection analysis data.
\newblock Journal of Chemometrics \textbf{6}(1), 29--40 (1992)

\bibitem{malinowski2002facanalchem}
Malinowski, E.R.: Factor Analysis in Chemistry, 3rd Edition.
\newblock John Wiley \& Sons (2002)

\bibitem{manne1995resolution}
Manne, R.: On the resolution problem in hyphenated chromatography.
\newblock Chemometrics and Intelligent Laboratory Systems \textbf{27}(1),
  89--94 (1995)

\bibitem{Neymeyr2019NMFg}
Neymeyr, K., Sawall, M.: On the set of solutions of the nonnegative matrix
  factorization problem.
\newblock SIAM Journal on Matrix Analysis and Applications \textbf{39}(2),
  1049--1069 (2019)

\bibitem{Omidikia2018closure}
Omidikia, N., Beyramysoltan, S., Mohammad~Jafari, J., Tavakkoli, E.,
  Akbari~Lakeh, M., Alinaghi, M., Ghaffari, M., Khodadadi~Karimvand, S.,
  Rajk{\'o}, R., Abdollahi, H.: Closure constraint in multivariate curve
  resolution.
\newblock Journal of Chemometrics \textbf{32}(12), e2975 (2018)

\bibitem{Omidikia2020sparse-norms}
Omidikia, N., Ghaffari, M., Rajk{\'o}, R.: Sparse non-negative multivariate
  curve resolution: {L}0, {L}1, or {L}2 norms?
\newblock Chemometrics and Intelligent Laboratory Systems \textbf{199}, 103,969
  (2020)

\bibitem{rajko2006duality}
Rajk{\'o}, R.: Natural duality in minimal constrained self modeling curve
  resolution.
\newblock Journal of Chemometrics \textbf{20}(3-4), 164--169 (2006)

\bibitem{rajko2009BorgenNorm}
Rajk{\'o}, R.: Studies on the adaptability of different borgen norms applied in
  self-modeling curve resolution ({SMCR}) method.
\newblock Journal of Chemometrics \textbf{23}(6), 265--274 (2009)

\bibitem{Rajko2010additional2comp}
Rajk{\'o}, R.: Additional knowledge for determining and interpreting feasible
  band boundaries in self-modeling/multivariate curve resolution of
  two-component systems.
\newblock Analytica Chimica Acta \textbf{661}(2), 129--132 (2010)

\bibitem{rajko2015definition}
Rajk{\'o}, R., Abdollahi, H., Beyramysoltan, S., Omidikia, N.: Definition and
  detection of data-based uniqueness in evaluating bilinear (two-way) chemical
  measurements.
\newblock Analytica Chimica Acta \textbf{855}, 21--33 (2015)

\bibitem{Rajko2005analsol3comp}
Rajk{\'o}, R., Istv{\'a}n, K.: Analytical solution for determining feasible
  regions of self-modeling curve resolution ({SMCR}) method based on
  computational geometry.
\newblock Journal of Chemometrics \textbf{19}(8), 448--463 (2005)

\bibitem{Sawall2016comp234ch}
Sawall, M., Jürß, A., Schröder, H., Neymeyr, K.: Chapter 5 - {On} the
  analysis and computation of the area of feasible solutions for two-, three-,
  and four-component systems.
\newblock In: C.~Ruckebusch (ed.) Resolving Spectral Mixtures, \emph{Data
  Handling in Science and Technology}, vol.~30, pp. 135--184. Elsevier (2016)

\bibitem{Sawall2013polinfl3}
Sawall, M., Kubis, C., Selent, D., B{\:o}rner, A., Neymeyr, K.: A fast polygon
  inflation algorithm to compute the area of feasible solutions for
  three-component systems. i: concepts and applications.
\newblock Journal of Chemometrics \textbf{27}(5), 106--116 (2013)

\bibitem{vavasis2010complexity}
Vavasis, S.A.: On the complexity of nonnegative matrix factorization.
\newblock SIAM Journal on Optimization \textbf{20}(3), 1364--1377 (2010)

\bibitem{Vosough2006gridsearch}
Vosough, M., Mason, C., Tauler, R., Jalali-Heravi, M., Maeder, M.: On
  rotational ambiguity in model-free analyses of multivariate data.
\newblock Journal of Chemometrics \textbf{20}(6-7), 302--310 (2006)

\bibitem{Wehrens2020R}
Wehrens, R.: Chemometrics with R: Multivariate Data Analysis in the Natural and
  Life Sciences (Use R!) 2nd ed.
\newblock Springer (2020)

\end{thebibliography}

\end{document}